\documentclass[a4paper,12pt]{article}
\usepackage{graphicx} 
\usepackage{float}
\usepackage[dvipsnames]{xcolor}
\usepackage{subfigure}
\usepackage{amsmath}
\usepackage{amsfonts}
\usepackage{amssymb}
\usepackage{amsthm}
\usepackage{tikz}

\newtheorem{theorem}{Theorem}

\newtheorem{lemma}[theorem]{Lemma}

\newtheorem{definition}[theorem]{Definition}

\usepackage{color}   
\usepackage{hyperref}
\hypersetup{
    colorlinks=true, 
    linktoc=all,     
    linkcolor=black,  
}

\begin{document}
\bibliographystyle{plain}
\title{Limiting Behavior of Randomly Perturbed Diffusions with Invariant Repelling Surfaces}
\author{
Leonid Koralov\footnote{Dept of Mathematics, University of Maryland,
College Park, MD 20742, koralov@umd.edu},
Chenglin Liu\footnote{Dept of Mathematics, University of Maryland, College Park,  MD 20742, clliu@umd.edu} 
}\date{ }\maketitle

\begin{abstract}
We study small non-degenerate perturbations of diffusions
in $\mathbb{R}^d$ with finitely many compact, disjoint,
repelling invariant hypersurfaces. These hypersurfaces
separate the state space into domains, each carrying
a unique invariant probability measure of the
unperturbed process. We establish precise power-law
asymptotics for the densities of the invariant measures of the unperturbed process  near the hypersurfaces, including the tangential profiles.
Using these asymptotics, we determine the relative exit
probabilities and mean exit times from the domains for the perturbed process.
We then identify a hierarchy of metastable distributions
associated with critical polynomial time scales.
Away from these critical scales, the limiting
distributions are convex combinations of the invariant
measures in the domains. We characterize their supports
through the domain adjacency tree and show that the
mixture weights depend on the initial domain and
observation time scale, but are independent of the
perturbation. Finally, the invariant measures
of the perturbed processes converge to the mixture
associated with the largest time scale.
\end{abstract}



\section{Introduction and main results}
\label{intro}
Consider a diffusion process $X_t$ that satisfies the stochastic differential equation
$$
d X_t = v_0(X_t)dt + \sum_{i=1}^d v_i(X_t) \circ dW_t^i,\quad X_0=x\in \mathbb{R}^d,
$$
where $W^i$ are independent Wiener processes and $v_0,v_1,..., v_d$ are $C^8(\mathbb{R}^d)$ vector fields. The Stratonovich form is convenient here since it allows one to provide a natural coordinate-independent description of the process.
The generator of $X_t$ is the operator
\begin{equation}
\label{degenerate operator}
L = L_0 + \frac{1}{2} \sum_{i=1}^d L_i^2,
\end{equation}
where $L_i u = \langle v_i,\nabla u\rangle$ is the operator of differentiation along the vector field $v_i,\, i=0,...,d$.
It should be kept in mind that the process depends on the initial point $x$, yet this is not
always reflected in the notation. We
will use the subscript $x$ to denote the probabilities and expectations associated with the
process starting at $x$.

Suppose that $S_1,...,S_m$ are $C^9$-smooth compact and connected non-intersecting $(d-1)$-dimensional surfaces in $\mathbb{R}^d$ that are invariant for each of the vector fields $v_0,v_1,...,v_d$. They divide the space into $m+1$ components $D_1,...,D_{m+1}$, with the union $\bigcup_{i=1}^{m+1} D_i$ denoted by $D$. Each surface  $S_k$ serves as a part of the boundary for two domains, say $D_i$ and $D_j$. In this case, we will alternatively denote $S_k$ as $S_{ij}$ (two indices are used to denote the surface that serves as an interface between two specific domains).  Let $S:=\bigcup_{k=1}^m S_k$.

We assume that $X_t$, restricted to a single surface, is a non-degenerate process. The diffusion matrix for $X_t$ is assumed to be  non-degenerate in each domain $D_i$. More precisely, let $T_{k}(x)$ be the tangent space to $S_{k}$ at $x$. We assume that, for each $i,k$,

(a) $\text{span}(v_0(x),v_1(x),\cdots, v_d(x))=\text{span}(v_1(x),\cdots, v_d(x)) = T_{k}(x),~x \in S_k$,

(b) $\text{span} (v_1(x),\cdots, v_d(x))= \mathbb{R}^d,~x\in D_i.$

Thus, the generator of the  process degenerates on each of the surfaces but not elsewhere. 
We will need another assumption that states that the process degenerates at the surfaces in a regular manner. Namely, by (\ref{degenerate operator}), the operator $L$ can be written as $Lu = \sum_{i,j=1}^d a^{ij}(x)\partial_{x_i x_j}u + \sum_{i=1}^d b^i(x) \partial_{x_i} u$ with $(x_1,...,x_d)\in \mathbb{R}^d$. Let $A(x):= (a^{ij}(x))_{ij}$. We make the following assumption.

(c) There is $c>0$ such that $\xi^T A(x) \xi \geq c\,\mathrm{ min}\{\mathrm{ dist}(x, S)^2, 1\}\,|\xi|^2,\, x\in \mathbb{R}^d,\, \xi \in \mathbb{R}^d$.

 In the next section, we will see that this assumption is an alternative but more convenient formulation of Assumption (e) in \cite{Metastability}. It can be seen to hold for a generic collection of vector fields satisfying Assumptions (a) and (b).

Under Assumption (a), the set $S$ is inaccessible, i.e., the process $X_t$ starting at
$x \in D$ does not reach it in finite time.
It was shown in \cite{Metastability} that each surface
can be classified as either attracting or repelling for $X_t$, depending (roughly speaking) on
whether $\mathrm{P}_x (\lim_{t \rightarrow \infty} \mathrm{dist}(X_t, S_k) = 0) > 0$ for each $x$ sufficiently close to $S_k$ (attracting
surface) or $\mathrm{P}_x (\lim_{t \rightarrow \infty} \mathrm{ dist}(X_t, S_k) = 0) = 0$ for each $x \notin S_k$ (repelling surface). Here, we will primarily concentrate on the situation where all the surfaces are repelling. 

It was shown in \cite{Metastability} that a real number $\gamma_k$ can be associated to each of the invariant surfaces, which characterizes the transition times and transition probabilities near the surfaces.
For example, if $S_k$ is repelling, $\gamma_k < 0$, and $U_k$ is a sufficiently small neighborhood of $S_k$, then
the probability that the process $X_t$ starting at a fixed point $x \in U_k \setminus S_k$
reaches the $r$-neighborhood of $S_k$ before exiting $U_k$ is of order $r^{-\gamma_{k}}$ when $r\downarrow 0$.

The precise way to determine $\gamma_k$ will be discussed in Section \ref{procbeh}; here we mention that $\gamma_k > 0$ implies that the surface is attracting, and $\gamma_k < 0$ implies that the surface is repelling (we do not consider $\gamma_k = 0$ for simplicity). Here, we stress that the numbers $\gamma_k$ are different from the normal Lyapunov exponents: the latter describe typical normal growth (for repelling surfaces), whereas $\gamma_k$ govern the power-law order of rare approaches to $S_k$ (see the remark after Lemma \ref{lemmaspectral}).

Next, we perturb the process $X_t$ by a small non-degenerate diffusion. The resulting process $X_t^\varepsilon$ satisfies stochastic differential equation
$$
dX_t^\varepsilon = (v_0 + \varepsilon^2 \widetilde v_0) (X_t^\varepsilon) dt + \sum_{i=1}^d v_i(X_t^\varepsilon)\circ dW_t^i + \varepsilon \sum_{i=1}^d \widetilde v_i (X_t^\varepsilon) \circ d\widetilde W_t^i
$$
with $\widetilde W_t^i$ independent Wiener processes (also independent with $W_t^i$) and $\widetilde v_0,\widetilde v_1,\cdots, \widetilde v_d$ bounded $C^8(\mathbb{R}^d)$ vector fields. The precise assumption on non-degeneracy of $\widetilde L$ is stated as follows:

(d) $\text{span}(\widetilde v_1, \cdots, \widetilde v_d)= \mathbb{R}^d~\mathrm{ for}~ x\in \mathbb{R}^d$.

We also include an assumption to control the behavior
of the processes at infinity. Recall that
$b=(b^1,\ldots,b^d)$ is the first-order coefficient of $L$
introduced before Assumption (c), namely, $b=v_0+\frac12\sum_{i=1}^d(Dv_i)v_i$.

(e) The vector fields
$v_1,\ldots,v_d,\widetilde v_0,\ldots,\widetilde v_d$
and the first order derivatives of
$\widetilde v_1,\ldots,\widetilde v_d$ are bounded.
There exist constants $c,R>0$ such that
\[
\langle b(x),x\rangle\leq-c|x|,
\qquad |x|\geq R.
\]
The generator of the process $X_t^\varepsilon$ is the operator $L^\varepsilon = L + \varepsilon^2 \widetilde L$ with
$$
\widetilde L  = \widetilde L_0 + \frac{1}{2} \sum_{i=1}^d \widetilde L_i^2.
$$
Here, $\widetilde{L}_i u = \langle \widetilde v_i, \nabla u\rangle$ is the operator of differentiation along the vector field $\widetilde v_i$, $i=0,...,d$.

At finite time scales, 
the trajectories of the process $X^\varepsilon_t$ are close to those of $X_t$.
However, when $t$ is large, i.e., $t = t(\varepsilon) \rightarrow \infty$ as $\varepsilon \downarrow 0$,
the invariant surfaces start playing an important role. Namely, $X_t$ may get forever trapped in a vicinity of an attracting surface, while $X^\varepsilon_t$ will escape from such a neighborhood after a long ($\varepsilon$-dependent) time due to the presence of small non-degenerate diffusion. On the other hand, $X_t$ cannot cross a repelling hypersurface, while $X^\varepsilon_t$ can cross, again due to the small non-degenerate term. 

In \cite{Metastability}, polynomial transition scales
and metastable limits were established for randomly
perturbed diffusions with invariant manifolds.
That work leaves the case of separating
repelling hypersurfaces for subsequent study.
The present paper addresses this disconnected-domain
setting. The invariant hypersurfaces divide the state
space into distinct domains, each carrying an invariant
measure of the unperturbed process. Although the
unperturbed process cannot move between these domains,
a small non-degenerate perturbation permits rare
crossings of their common boundaries.
Using precise boundary asymptotics of the invariant
densities, we determine the relative probabilities
of these crossings and identify the hierarchy of
communicating groups of domains on successively longer
time scales. The corresponding metastable distributions
are linear combinations of the invariant measures
in these domains. We determine their supports and
show that the mixture weights depend on the initial
domain and the observation time scale, but are
independent of the perturbation operator $\widetilde L$.

Let $\mu_i$ denote the invariant measure for the unperturbed process $X_t$ in the domain $D_i$, $1 \leq i \leq m+1$. The existence and uniqueness of such a measure in a bounded domain $D_i$ have been demonstrated in \cite{freidlin2023perturbations}. When $D_i$ is unbounded, the existence and uniqueness follow from a similar argument as in \cite{freidlin2023perturbations} and Assumption (e).
A principal analytical contribution of the present
paper is the precise boundary asymptotic of these
measures (see Theorem~\ref{invariant measure multi-surfaces}).
Near each hypersurface $S\subseteq\partial D_i$, for some $c>0$, we show that the density $f$ of the invariant measure $\mu_i$ has the form
\[
f(y,z)=c\,\psi(y)z^{-\gamma-1}(1+o(1)),
\quad \text{as}\,z\downarrow0,\quad \text{uniformly in}\, y\in S,
\]
where $z$ is the distance to $S$ and $\psi$ is a positive
solution of an adjoint equation on $S$. This result identifies the tangential profile
and the power governing the mass near the boundary. In Section~\ref{Probability and expectation time for transition in one domain}, we develop an excursion argument that converts these boundary asymptotics into precise exit probabilities. We estimate the occupation-time to determine the relative frequencies of visits to different boundary neighborhoods, which are then combined with local boundary-hitting estimates. The analysis also controls exit times conditioned on each boundary component, including those with vanishing exit probabilities, and yields a common leading mean exit time and exponential limits for the subsequent semi-Markov analysis (see Lemma \ref{abstract game}).

Without loss of generality, let us re-label the surfaces in such a way that their characteristic numbers $\gamma_k$ are listed in the decreasing order: $\gamma_m \leq \gamma_{m-1} 
 \leq ... \leq \gamma_1 < 0$. Let us also put $\gamma_0 = 0$ and $\gamma_{m+1} = -\infty$.
 We will describe the behavior of $X^\varepsilon_{t(\varepsilon)}$ assuming that $t(\varepsilon)$ satisfies 
 $\varepsilon^{\gamma_{k-1}} \ll t(\varepsilon) \ll \varepsilon^{\gamma_{k}}$, $1 \leq k \leq m+1$.
 Here, the relation $f(\varepsilon) \ll g(\varepsilon)$ between two positive functions means that $f(\varepsilon)/g(\varepsilon) \rightarrow 0$ as $\varepsilon \downarrow 0$. The notation $f(\varepsilon)\sim g(\varepsilon)$ for two positive functions means that $f(\varepsilon)/g(\varepsilon)\to 1$ as $\varepsilon\downarrow 0$. The main result can be stated as follows.
{ \begin{theorem}
\label{main theorem}
Suppose that Assumptions (a)--(e) hold and that
all surfaces $S_k$, $1\leq k\leq m$, are repelling (with $\gamma_k < 0$). There exist positive numbers $w_1,\ldots,w_{m+1}$,
determined solely by the coefficients of $L$, with
$\sum_{j=1}^{m+1}w_j=1$, such that the following holds.

For $i,k\in\{1,\ldots,m+1\}$ with
$\gamma_{k-1}>\gamma_k$, let $\mathcal I_k(i)$
be the set of indices $j$ such that $D_i$ and $D_j$
can be connected by a path of adjacent domains
whose separating surfaces have critical values
strictly greater than $\gamma_k$. We have
\[
\mathrm{Law}_x(X^\varepsilon_{t(\varepsilon)})
\longrightarrow
\sum_{j=1}^{m+1}c_j(i,k)\mu_j,
\qquad \varepsilon\downarrow0,
\]
uniformly for $x$ in compact subsets of $D_i$, for all $t(\varepsilon)$ such that $\varepsilon^{\gamma_{k-1}}\ll t(\varepsilon)\ll \varepsilon^{\gamma_{k}}$. Here
\[
c_j(i,k)=
\begin{cases}
\displaystyle
\frac{w_j}{\sum_{l\in\mathcal I_k(i)}w_l},
&j\in\mathcal I_k(i),\\[6pt]
0,&j\notin\mathcal I_k(i).
\end{cases}
\]
In particular, the coefficients are independent of
the perturbation $\widetilde L$.

The limit of the invariant measure for the process $X^\varepsilon_t$ coincides with the metastable distribution at the largest time scale, i.e., the invariant measure for $X^\varepsilon_t$ converges to $\sum_{j =1}^{m+1} c_j \mu_j$ as $\varepsilon \downarrow 0$, where the coefficients $c_j = c_j(i,m+1) > 0$ do not depend on $i$. 
\end{theorem}}

The paper is organized as follows. In Section \ref{procbeh}, we introduce the expansion of the operator $L$ and the definition of constants $\gamma_k$ associated to the surface $S_k$. We also collect the  necessary estimates for the transition behavior of the processes $X_t, X_t^\varepsilon$, most of which have been established in \cite{freidlin2023perturbations}, \cite{Metastability}. In Section~\ref{invmeas}, we establish precise boundary asymptotics for the invariant measures $\mu_j$ of the unperturbed process. In Section~\ref{Probability and expectation time for transition in one domain}, we provide precise estimates on the transition probabilities and transition times for $X_t^\varepsilon$ to hit the surfaces. In Section~\ref{Exit from surfaces}, we describe how $X_t^\varepsilon$, starting on some surface $S_k$, escapes a neighborhood of $S_k$. In Section \ref{vague version main result}, we prove Theorem \ref{main theorem} by identifying the constants $c_j(i,k)$ in terms of a limiting invariant measure for a semi-Markov process and showing that $c_j(i,k)$ can be determined in terms of the coefficients of $L$.

We now outline how the current work is related to other research. The metastable behavior in the present model can be compared to that arising in small random perturbations of dynamical systems. Freidlin–Wentzell theory (see  \cite{freidlin2012random}) was among the first rigorous results for metastability in stochastic systems. 
It relies on large deviations principle that gives the asymptotics of transition times and transition probabilities between the vicinities of different equilibria. Such results were refined in some cases: for example, for random perturbations that result in reversible diffusion processes, precise asymptotics of the expected transition times (the Eyring-Kramers formula) have been established in \cite{sharpexittimeestimate},\cite{asymptoticexittime}. Moreover, for a more general class of randomly perturbed dynamical systems, metastability  was recently studied in \cite{landim2024metastability}, \cite{landim2025metastabilitytimescalesparabolic}. Metastability for randomly perturbed stochastic systems with invariant surfaces (some of which are attracting) has recently been studied in \cite{freidlin2023perturbations}, \cite{Metastability}. The main distinction in our setting is that,
with all surfaces repelling, the invariant measures of the unperturbed process are supported in full-dimensional domains, requiring a different analysis of transition times and probabilities. It should be stressed that in the Freidlin-Wentzell theory, the transition times scale exponentially in the size of the perturbation. In the current work (as well as in \cite{freidlin2023perturbations}, \cite{Metastability}), they scale as powers of $\varepsilon$. This is not surprising: the fact that the unperturbed system is stochastic means that the main contribution to the transition times comes only from the 
time spent in the vicinity of the surfaces of degeneration (in contrast to the Freidlin-Wentzell case, where large deviations are needed to overcome the deterministic drift).

Classical results of \cite{BaxendaleStroock1988} and \cite{Baxendale1994} establish two-sided power-law bounds for invariant mass for certain systems (near an invariant diagonal in a two-point motion and near a fixed point, respectively). The exponents in these bounds are determined by nonzero roots of the corresponding moment Lyapunov exponent functions. More recently, \cite{bakhtin2025random} studied the limiting behavior of empirical measures for 2-dimensional diffusions with degeneracies at the boundary, providing a complementary perspective on their long-time statistical behavior. In the present setting of smooth repelling invariant hypersurfaces, we establish the sharper density asymptotics $f(y,z)\sim c\psi(y)z^{-\gamma-1}$ as $z\downarrow0$, identifying both the tangential profile and the exponent. These precise asymptotics also allow us to determine the relative probabilities of exits through different boundary components.

Related classes of degenerate diffusions also arise in the study of stochastic persistence. In stochastic population models, the boundary faces are often invariant, and near-boundary behavior plays an important role in establishing criteria for persistence and extinction; see \cite{SchreiberBenaimAtchade2011,HeningNguyen2018}. Degenerate diffusion systems arise naturally in a variety of applications, for example, continuous-time measurement in quantum mechanics (\cite{barchielli2009quantum}, \cite{attractingquantum}, \cite{chalal2023meanfield}), population biology (\cite{belabbas2021preypredator}), and scalable reaction networks (see \cite{lin2020network} and the references therein). Degenerate diffusions with repelling surfaces arise in the description of near-wall particle motion in viscous fluids (\cite{hydrodynamic}).

The general study of metastability, as well as the analysis of some specific models, can be found in \cite{berglund2025reducing}, \cite{bovier2016metastability},   \cite{olivieri2005metastability}, \cite{landim2016metastablemarkovchain}, \cite{betz2016multiscale} (this is not an exhaustive list). In  Section~\ref{vague version main result} of this paper,  we make use of abstract results  from  
\cite{koralov2024metastable} on metastability for semi-Markov processes.

\section{Behavior of the processes \texorpdfstring{$X_t$}{TEXT} and \texorpdfstring{$X^\varepsilon_t$}{TEXT} and of the corresponding operators near \texorpdfstring{$S_k$}{TEXT}}
\label{procbeh}
The behavior of the processes near the invariant surfaces has been studied in \cite{freidlin2023perturbations} and \cite{Metastability}.
We need to gather these facts in order to describe the critical exponents $\gamma_k$ and in order to have the needed tools to study the asymptotics of the transition times and transition probabilities. In this section,  we largely follow \cite{freidlin2023perturbations} and \cite{Metastability}.

Let us fix $S = S_k$, $1 \leq k \leq m$, and $D = D_i$, $1 \leq i \leq m+1$, such that $S \subseteq \partial D$. We will study the behavior of the processes in a small neighborhood of $S$ in $D$. Each point $x$ in a sufficiently small such neighborhood can be uniquely identified with the pair $(y,z)$, where $y$ is the point on the surface that is closest to $x$ and $z$ is the distance of $x$ from the surface. Thus, for sufficiently small $\delta > 0$, there is a bijection between a small neighborhood of $S$ in $D$ and the set $S^\delta = S \times [0,\delta)$. The operator $L$ (which is the generator of $X_t$) can be considered in the $(y,z)$ coordinates. Let $L_y$ be the restriction of the operator $L$ to functions that are defined on $S$. This operator can be applied to functions of $(y,z)$ by treating $z$ as fixed.
 
 The following two lemmas describe the expression for the operator $L$  around $S$ and the definition of the characteristic number $\gamma = \gamma_k$ corresponding to $S = S_k$.
\begin{lemma} \label{loccoeff} For sufficiently small $\delta >0$,
the operator $L$ in $S^\delta$ can be written as:
$$
L u = L_y u + z^2 \alpha(y)\frac{\partial^2 u}{\partial z^2} + z\beta(y)\frac{\partial u }{\partial z} + z \mathcal{D}_y \frac{\partial u}{\partial z} + Ru
$$
with $R u = z \mathcal{K}_y u + z^2 \mathcal{N}_y \frac{\partial u}{\partial z} + z^3 \sigma(y,z)\frac{\partial^2 u}{\partial z^2}$,
where $\mathcal{D}_y$ is a first-order differential operator in $y$ (without a zero-order term), whose coefficients depend only on the $y$ variables, $\mathcal{K}_y$ is a differential operator on $S^\delta$ with
first- and second-order derivatives in $y$, $\mathcal{N}_y$ is a differential operator on $S^\delta$ with first-order derivatives in $y$ and a potential term. All the operators have $C^5$ coefficients, while $\alpha,\beta\in C^5(S), \, \sigma\in C^5(S^\delta)$.
\end{lemma}

It is easy to see, using the smoothness of the vector fields $v_i$ or from the proof of Lemma~\ref{loccoeff} in \cite{freidlin2023perturbations}, that the functions $\alpha$ and $\beta$ and the coefficients of the operator $\mathcal{D}_y$ do not depend on the choice of the side of the surface $S$ ($S^\delta$ was defined by the pair $S$ and $D$).
This lemma was proved in \cite{freidlin2023perturbations} by considering the leading terms in the Taylor expansion of each of the vector fields $v_i$ around the surface $S$.

It follows from Assumption (a) that the process $X_t$ has a unique invariant probability  measure $\pi$ on $S$. By elliptic regularity, $\pi$ has a strictly
positive $C^6$ density with respect to surface
volume. Define
$$\alpha^{avg}= \int_S \alpha d\pi,~~\beta^{avg} =  \int_S \beta d\pi.$$
We have the following result.
\begin{lemma} \label{lemmaspectral}
If $\alpha^{avg} < \beta^{avg}$ ($\alpha^{avg} > \beta^{avg}$), then there exists $\gamma<0$ ($\gamma>0$) and a positive-valued $\varphi\in C^6(S)$ satisfying $\int_S \varphi d\pi =1$ such that
\begin{equation} \label{elliptic1}
L_y\varphi + \alpha(\gamma-1)\gamma\varphi+ \beta \gamma\varphi + \gamma \mathcal{D}_y \varphi = 0,
\end{equation}
and such $\gamma$ and $\varphi$ are determined uniquely.
\end{lemma}
For example, if $\alpha \equiv \alpha^{avg}$ and $\beta \equiv \beta^{avg}$ are constants, then $\gamma = 1-( \beta^{avg}/\alpha^{avg})$ and $\varphi \equiv 1$. In the general case, the lemma was proved in \cite{freidlin2023perturbations} and served as the basis for most of the analysis there. We are interested in the case when
$
\alpha^{avg}-\beta^{avg}<0$, i.e., 
 $\gamma<0$.
\\
\\
\textbf{Remark.} Let $\Lambda(\gamma)$ denote the principal eigenvalue of this operator
\[
f\longmapsto L_yf+\alpha \gamma(\gamma-1)f+\beta \gamma f+\gamma\mathcal D_yf.
\]
The function $\Lambda$ is convex and satisfies $\Lambda(0)=0$. Its derivative
\[
\Lambda'(0)=\int_{S}(\beta-\alpha)\,d\pi
\]
is the normal Lyapunov exponent, while $\gamma$ is its unique nonzero root. In the repelling case, $\Lambda'(0)>0$ and $\gamma<0$. The normal Lyapunov exponent describes typical normal growth, whereas $\gamma$ governs the power-law order of rare approaches to $S$.

It follows from the lemma that the function $\varphi(y)z^{\gamma}$, satisfies the equation
$$
M (\varphi(y)z^\gamma) = L_y \varphi(y)z^\gamma+z^2\alpha(y)\frac{\partial^2 \varphi(y)z^\gamma}{\partial z^2} + z \beta(y) \frac{\partial \varphi(y)z^\gamma}{\partial z} + z \mathcal{D}_y \frac{\partial \varphi(y)z^\gamma}{\partial z}  = 0
$$
in  $S^\delta$, where $M := L - R$ and $\delta$ is sufficiently small. With the change of variables $(y,\mathbf{z})=(y,\ln(z))$, the operator $M$ is transformed into
$$
Ku:= L_y u + \alpha(y) \frac{\partial^2 u}{\partial \mathbf{z}^2} + (\beta(y)-\alpha(y)) \frac{\partial u}{\partial \mathbf{z}} + \mathcal{D}_y \frac{\partial u}{\partial \mathbf{z}}.
$$
In \cite{Metastability}, it is assumed that the operator $K$ is elliptic on $S\times \mathbb{R}$. Note that the ellipticity of $K$ follows from
Assumption (c) in Section \ref{intro}.

For $\varkappa \geq 0$ and  $0 \leq \varkappa_1 \leq \varkappa_2$,  we can define the following sets:
%
\begin{equation}
\label{defineition of layers}
\begin{aligned}
\Gamma_\varkappa &= \{(y,z):(\varphi(y))^{\frac{1}{\gamma}}z=\varkappa\},\\
V_{\varkappa_1,\varkappa_2} &= \{(y,z): \varkappa_1\leq (\varphi(y))^{\frac{1}{\gamma}}z\leq \varkappa_2\}.\\
\end{aligned}
\end{equation}
These are subsets of $S^\delta$ if $\varkappa, \varkappa_1$, and $\varkappa_2$ are sufficiently small.
The transition probabilities and transition times between different layers $\Gamma_\varkappa$ for the processes $X_t$ and $X^\varepsilon_t$ have been estimated in \cite{freidlin2023perturbations} and \cite{Metastability}, where the proofs of the following lemmas can  be found. One of the main ideas 
involved in the proofs is that the process $\varphi(Y_t) Z_t^\gamma$ is almost a martingale in a neighborhood of $S$ due to the choice of $\gamma$ (Lemma~\ref{lemmaspectral}). 
Had it been a martingale, the transition probabilities could have been written out explicitly using the optional sampling theorem.
What makes it different from a martingale is the presence of the correction term $R$ in the generator $L$ of the process $X_t$ (see Lemma~\ref{loccoeff}). Appropriate sub- and super-solutions (closely approximating the actual solutions) of the corresponding elliptic problems were used to estimate the desired quantities.

We will use the notation $\tau(A)$ ($\tau^\varepsilon(A)$) for the first hitting time of a set $A$ by the process $X_t$ ($X^\varepsilon_t$).

\begin{lemma}
\label{unperturbed transition probability}
Let $\gamma<0$. For each $\eta>0$, for all sufficiently small $\varkappa_1,\varkappa_2$ (depending on $\eta$) satisfying $0<\varkappa_1<\varkappa_2$,
$$
\frac{(1+\eta)\zeta^\gamma-\varkappa_1^\gamma}{\varkappa_2^\gamma-\varkappa_1^\gamma} \leq \mathrm{P}_x(X_{\tau(\Gamma_{\varkappa_1}\cup \Gamma_{\varkappa_2})}\in \Gamma_{\varkappa_2})\leq \frac{(1-\eta)\zeta^\gamma-\varkappa_1^\gamma}{\varkappa_2^\gamma-\varkappa_1^\gamma} 
$$
provided that $x \in  \Gamma_\zeta$ with $\varkappa_1\leq \zeta \leq \varkappa_2$.
\end{lemma}
A similar probability estimate holds for $X_t^\varepsilon$ with $\varepsilon$ small.
\begin{lemma}
\label{perturbed transition probability}
Let $\gamma<0$. For each $\eta>0$, for all sufficiently large $r$ (depending on $\eta$), for all sufficiently small $\varkappa_2>0$ (depending on $\eta$), for all sufficiently small $\varepsilon$ (depending on $\eta$, $r$, and $\varkappa_2$), and for all $\varkappa_1\in[r\varepsilon,\varkappa_2)$,
$$
\frac{(1+\eta)\zeta^\gamma - \varkappa_1^\gamma}{\varkappa_2^\gamma-\varkappa_1^\gamma}\leq \mathrm{P}_x(X^{\varepsilon}_{\tau(\Gamma_{\varkappa_1}\bigcup \Gamma_{\varkappa_2})}\in \Gamma_{\varkappa_2}) \leq \frac{(1-\eta)\zeta^\gamma - \varkappa_1^\gamma}{\varkappa_2^\gamma-\varkappa_1^\gamma}
$$
provided that $x\in \Gamma_\zeta$ with $\varkappa_1\leq \zeta\leq \varkappa_2$.
\end{lemma}
The perturbed process $X_t^\epsilon$ has small probability to reach a repelling surface $S$.
\begin{lemma}
\label{transition probability to surface}
Let $\gamma<0 $. There is $\rho>0$ such that, for each $\eta >0$, for all
sufficiently small $\varkappa$, sufficiently large $s_1>0$, and sufficiently small $s_2>0$ (all dependent of $\eta$), there is  $\varepsilon_0$ such that
$$
\rho(1-\eta) \frac{\zeta^\gamma}{\varepsilon^\gamma}\leq \mathrm{P}_x(X^{\varepsilon}_{\tau(S\bigcup \Gamma_\varkappa)}\in S) \leq \rho(1+\eta) \frac{\zeta^\gamma}{\varepsilon^\gamma}
$$
provided that $x\in \Gamma_\zeta$ with $s_1\varepsilon \leq \zeta \leq s_2\varkappa$ and $0\leq \varepsilon\leq \varepsilon_0$.
\end{lemma}
Estimates on the expectation of the exit time have also been obtained.
\begin{lemma}
\label{unperturbed expectation exit time}
Let $\gamma<0$. There is $b>0$  such that, for all sufficiently small $\varkappa>0$, 
\[
\mathrm{E}_x\tau(\Gamma_\varkappa) \leq b(1 + \ln(\varkappa/\zeta)),~x \in V_{0,\varkappa}\setminus S,
\]
where $\zeta= \zeta(x)$ is such that $x\in \Gamma_{\zeta}$.
\end{lemma}
The following Lemma is slightly different from the original version in \cite{Metastability}. However, the proof is similar.
\begin{lemma}
\label{perturbed expectation time}
Let $\gamma<0$. There is $b>0$ and for all sufficiently small $\varkappa$ there is $\varepsilon_0>0$ (that depends on $\varkappa$) such that
$$
\mathrm{E}_x\tau^\varepsilon( S\bigcup \Gamma_\varkappa) \leq b(1 + \min(\ln(\varkappa/\varepsilon), \ln(\varkappa/\zeta))),~x \in V_{0,\varkappa}, ~0 < \varepsilon \leq \varepsilon_0,
$$
where $\zeta= \zeta(x)$ is such that $x\in \Gamma_{\zeta}$.
\end{lemma}

We will also need some facts on the behavior of non-degenerate processes and the solutions to the corresponding PDEs in a given domain. 

\begin{lemma} \label{ellipticity estimation}

Let $U$ be an open connected domain in $\mathbb{R}^d$ or in a $d$-dimensional manifold embedded in a Euclidean space. Suppose that $\partial U = B= B_1 \bigcup ... \bigcup B_n$ is a disjoint union of a finite number of smooth $(d-1)$-dimensional compact surfaces and suppose that $\overline{U} = U \bigcup B$ is a compact set. Let an operator $L$ with smooth coefficients be uniformly elliptic on $\overline{U}$. Let $K \subset U$ be a compact set. Let $X_t$ be the diffusion process generated by $L$ (without the potential term). Let $\tau = \inf\{t: X_t \in B \}$.

$(i)$ Let $\mu^x$ be the probability measure on $B$ induced by $X_{\tau}$ starting at $x$, and let $p^x$ be its density with respect to the Lebesgue measure on $B$. Then there exists $C >0$ such  
that 
$$
p^x(\tilde{x}) \geq C,~~x \in K,~\tilde{x} \in B.
$$

$(ii)$ There exists $C>0$ such that
$$
\mathrm{E}_x\tau \geq C,\,x\in K.
$$

$(iii)$ Suppose that a non-negative  function $g \in C(\overline{U})$ satisfies  $L g(x) = 0,\, x\in U. $
Then there exists $C>0$ such that
$$
\sup_{K} g \leq C \inf_{K} g.
$$

$(iv)$ Suppose that $B$ is $C^2$-smooth and a non-constant function $g \in C(\overline{U})$  satisfies $L g(x) = 0,\, x\in U$, where $L$ does not have a potential term. Suppose that $g(x) = 0,\, x\in B_1$, and
$g(x) = 1,\, x\in B_2\bigcup....\bigcup B_n.$ 
Then there exist $C, C'>0$ such that
$$
C\leq \inf_{B_1} \frac{\partial g}{\partial \nu} \leq \sup_{B_1} \frac{\partial g}{\partial \nu}\leq C'
$$\
where $\nu$ is the interior normal vector field.

$(v)$ Suppose that  $B$ is $C^2$-smooth and $g(x)\in C(\overline{U})$ satisfies $Lg(x) = 0,\, x\in U$. Then there exists $C>0$ such that
$$
\sup_{x\in K}|\nabla g(x)|\leq \frac{C}{\mathrm{ dist} (K, B)} \sup_{x\in \overline{U}} |g(x)|.
$$
In each statement, the constants can be chosen
uniformly over operators with a common positive
ellipticity lower bound and a common bound on the
$C^3(\overline{U})$-norms of their coefficients.
\end{lemma}
The statement in (i) is probably known in elliptic theory, its proof was sketched in \cite{Metastability}. The statement in (ii) is standard. The statements in (iii), (iv) and (v) are standard  in elliptic theory (see \cite{ellipticbible} and \cite{han2011elliptic}).
\\
\\
\textbf{Remark.}
Let the function $u^\varkappa(x)$ satisfy:
$$
Lu^\varkappa(x) = 0,\quad x=(y,z) \in \mathrm{ int}\,V_{\varkappa,2\varkappa}
$$
with the operator $L$ that appears in Lemma \ref{loccoeff}. In the coordinate $\hat{x}=(y,\hat z)\in \mathrm{ int}\,V_{\varkappa_1,2\varkappa_1}$ with $\hat{z} = \frac{\varkappa_1}{\varkappa}z$, let
$$
\hat{L} := L_y + \hat{z}^2 \alpha(y)\partial^2_{\hat{z}} + \hat{z}\beta(y)\partial_{\hat{z}} + \hat{z} \mathcal{D}_y \partial_{\hat{z}}+ \frac{\varkappa}{\varkappa_1}\hat{R}
$$
and
$$
\hat{R}:= \hat{z} \hat{\mathcal{K}}_y + \hat{z}^2 \hat{\mathcal{N}}_y \partial_{\hat{z}}+ \hat{z}^3 \sigma(y,\frac{\varkappa}{\varkappa_1}\hat{z})\partial_{\hat{z}}^2.
$$
Here, $\hat{\mathcal{K}}_y$ and $\hat{\mathcal{N}}_y$ are operators similar to $\mathcal{K}_y$ and $\mathcal{N}_y$ defined in Lemma \ref{loccoeff}, while they have the parameter $z$ replaced by $\frac{\varkappa}{\varkappa_1}\hat{z}$ in the coefficients of the operators $\mathcal{K}_y$ and $\mathcal{N}_y$, respectively. Consequently, $\hat{u}^\varkappa(\hat{x}):=u^\varkappa(y,\frac{\varkappa}{\varkappa_1}\hat{z})$ satisfies
\begin{equation}
\label{new elliptic operator}
    \hat{L}\hat{u}^\varkappa(\hat{x}) = 0,\quad \hat{x}\in \mathrm{ int}\, V_{\varkappa_1,2\varkappa_1}.
\end{equation}
Note that the bound on the coefficients in $\frac{\varkappa}{\varkappa_1}\hat{R}$ can be made arbitrarily small by taking $\frac{\varkappa}{\varkappa_1}$ sufficiently small (e.g., fix $\varkappa_1$ and take $\varkappa$ sufficiently small). We will repeatedly apply Lemma \ref{ellipticity estimation} to (\ref{new elliptic operator}) for a fixed $\varkappa_1$ and for all sufficiently small $\varkappa$ to obtain bounds with a constant $C$ that does not depend on $\varkappa$.

\section{The structure of the invariant measure for \texorpdfstring{$X_t$}{TEXT} in a single domain}
\label{invmeas}

 Since the boundaries of the domains $D_i$, $1 \leq i \leq m+1$, are inaccessible, the unperturbed process $X_t$ starting in $D_i$ will remain in the same domain. Therefore, we can temporarily drop the subscript $i$ and focus on the behavior of $X_t$ in one domain $D = D_i$.

 Some of the surfaces $S_1,...,S_m$ serve as the components of the boundary for $D$. We will write
$\partial D = S^D_1 \bigcup ... \bigcup S^D_{m_D}$, where 
$\{ S^D_1,...,S^D_{m_D}\} \subseteq \{S_1,...,S_m\}$. For a bounded domain $D$, since the surfaces $S^D_1,...,S^D_{m_D}$ are all repelling, i.e., $\gamma^D_1,...,\gamma^D_{m_D}$ are negative, the process $X_t$ has a unique invariant measure $\mu$, as shown in Theorem 5.6 in \cite{freidlin2023perturbations}. As mentioned in Section \ref{intro}, the existence and uniqueness of an invariant measure $\mu$ in an unbounded domain $D$ follow by an argument similar to Theorem 5.6 in \cite{freidlin2023perturbations} and by Assumption (e).

In this section, the asymptotic behavior of $\mu$ near each of the repelling surfaces $S = S_i^D, \, 1\leq i \leq m_D$, is described.
The main result of this section will be obtained for each surface $S^D_i$ separately, and thus we can temporarily drop the sub- and super-scripts, and denote the surface $S = S^D_i$.

Let $L_y^*$ and $\mathcal{D}_y^*$ be the adjoint operators to $L_y$ and $\mathcal{D}_y$ with respect to the measure $\pi$ on $S$.
\begin{lemma}
\label{adjointelliptic}
There exists a unique positive function $\psi\in C^6(S)$ such that $\int_S \psi d\pi = 1$ and 
$$
    L_y^* \psi + (\gamma-1)\gamma\alpha \psi +\gamma \beta \psi + \gamma \mathcal{D}_y^* \psi = 0.
$$
\end{lemma}
\begin{proof}
The operator $L_y^* +
(\gamma-1)\gamma\alpha +\gamma \beta + \gamma \mathcal{D}^*_y$ is adjoint to the operator on the left-hand side of (\ref{elliptic1}). Thus, it has the same principal eigenvalue (equal to zero) and the unique, up to a multiplicative constant,  positive eigenfunction corresponding to this eigenvalue. We denote this eigenfunction by $\psi$. Note that under the standing regularity assumptions,
elliptic regularity gives $\psi\in C^6(S)$.
\end{proof}
Recall that $M = L- R$. Let $M^*$ be the formal adjoint of $M$ in $S^\delta$ with respect to the product of the  measure $\pi$ on $S$ (for the $y$ variables) and the Lebesgue measure on the real line (for the $z$ variable).
By Lemma \ref{adjointelliptic}, the function $\psi(y)z^{-\gamma -1}$ satisfies the equation
\begin{equation}
\label{quasiadjoint}
\begin{aligned}
M^*( \psi(y) z^{-\gamma-1})= 0,~~ (y,z)\in S^\delta
\end{aligned}
\end{equation}
for sufficiently small $\delta>0$.

It was shown in \cite{freidlin2023perturbations} that the invariant measure $\mu$ (when restricted to $S^\delta$) has a density with respect to the product measure $d\pi\times dz$. We will denote this density  by $f$.  This non-negative function $f$ satisfies $\int_{S^\delta} f(y,z) L g(y,z)  d \pi(y) dz = 0,$ for any $g$ that is $C^2$-smooth and compactly supported in $S^\delta$.    Since the coefficients of the operator $L$ are smooth, it follows (see \cite{bogachev2001regularity},\cite{zhang2012regularity}) that $f$ is a strong solution to
\begin{equation}
\label{invariantmeasureoperator}
    L^*f(y,z) = 0,\quad (y,z)\in S^\delta,
\end{equation}
where $L^*$ is the formal adjoint of $L$.

The main theorem of this section describes the asymptotic behavior of $f$.
\begin{theorem}
\label{invariant measure}
There exists $c>0$ such that
$$
\lim_{\varkappa\downarrow 0} \frac{f(y,z)}{\psi(y) z^{-\gamma-1}} = c, \quad \mathrm{ uniformly}~\mathrm{ in}~ (y,z)\in \Gamma_\varkappa.
$$
\end{theorem}
With $\psi(y) z^{-\gamma-1}$ being the solution to $M^*$, 
the discrepancy between $\psi(y) z^{-\gamma-1}$ and $f$ is due to the presence of the term
$R^*$ (the formal adjoint to $R$) in $L^*$. Let us write
$$
f(y,z) = g(y,z)\psi(y)z^{-\gamma-1}
$$
with some unknown positive function $g(y,z)$. From (\ref{quasiadjoint}) and  (\ref{invariantmeasureoperator}), it follows that $g$ satisfies an equation with an operator that has a form similar to the operator $L$ that appears in Lemma \ref{loccoeff}. Some of the coefficients are different, though, and there is an additional small potential term. The equation that $g$ satisfies can be explained using the following operators. 

Let
\[
\overline{L}_yg
:=\psi^{-1}(L_y^*(\psi g)-gL_y^*\psi)
-\gamma\mathcal{D}_y g.
\]
Let
\begin{equation}
\label{defbarm}
\overline{M}g:=\overline{L}_y g+\alpha z^2\partial_z^2g+(2(1-\gamma)\alpha-\beta
-\frac{\mathcal{D}_y^*\psi}{\psi})z\partial_z g
+z\mathcal{D}_y\partial_z g.
\end{equation}
Using Lemma~\ref{adjointelliptic} and the identity
$\mathcal{D}_y^*(\psi g)
=g\mathcal{D}_y^*\psi-\psi\mathcal{D}_y g$,
a direct calculation gives
\[
\overline{M}g
=z^{\gamma+1}\psi^{-1}
M^*(\psi(y) z^{-\gamma-1}g).
\]
This operator has a similar structure to $M$ (a sum of an elliptic operator in the $y$ variables and a term that is homogeneous in the $z$ variables). Let $\overline{T} g $ be the potential term in the operator ${z^{\gamma+1}\psi^{-1}(y)}R^*(g(y,z)\psi(y)z^{-\gamma-1})$. It has the form
$$
\overline{T} g = z T(y,z) g(y,z)
$$
with $T(y,z)\in C^3(S^\delta)$. The remainder  will be denoted by $\overline{R}$:
\[
\overline{R} g := {z^{\gamma+1}\psi^{-1}(y)}R^*(g(y,z)\psi(y)z^{-\gamma-1}) - \overline{T}  g.
\]
By the above definitions, from (\ref{invariantmeasureoperator}), it follows that  $g$ satisfies
\begin{equation} 
\label{eqnbsr}
\overline{L}g := \overline{M}g + \overline{R}g + \overline{T} g = 0,
\end{equation}
where $\overline{M}$ can be written similarly to $M$:
\[
\overline{M} u = \overline{L}_y u + z^2 \overline{\alpha}(y)\frac{\partial^2 u}{\partial z^2} + z\overline{\beta}(y)\frac{\partial u }{\partial z} + z \overline{\mathcal{D}}_y \frac{\partial u}{\partial z}.
\]
Note that $\overline{\alpha} = \alpha,\,\overline{\mathcal{D}}_y= \mathcal{D}_y$, while $\overline{\beta}$ can be found in (\ref{defbarm}). Note that the coefficients of the operator $\overline{M}$ are in $C^3(S^\delta)$. From our construction it follows that $\overline{R}$ has similar properties to the operator $R$ in Lemma~\ref{loccoeff}. Namely,
$\overline{R} u = z \overline{\mathcal{K}}_y u + z^2 \overline{\mathcal{N}}_y \frac{\partial u}{\partial z} + z^3 \overline{\sigma}(y,z)\frac{\partial^2 u}{\partial z^2}$, where $\overline{\mathcal{K}}_y$ is a differential operator on $S^\delta$ with
first- and second-order derivatives in $y$, $\overline{\mathcal{N}}_y$ is a differential operator on $S^\delta$ with first-order derivatives in $y$ and a potential term. All the operators have $C^3$ coefficients, while $\overline{\alpha},\overline{\beta}\in C^3(S),\,\overline{ \sigma}\in C^3(S^\delta)$. Let $\overline{X}_t= (\overline{Y}_t,\overline{Z}_t)$ be the diffusion generated by the operator $\overline{M}+\overline{R}$.

Now, to prove Theorem \ref{invariant measure}, it is sufficient to show  that $g$ is asymptotically constant near the surface $S$. This is achieved in the following steps.
First, we show that $g$ satisfies a uniform bound in a neighborhood of $S$.
\begin{lemma}
\label{uniform bound on g}
There are positive constants $c$ and $C$ such that
\begin{equation} 
\label{boundsong}
c\leq g(y,z)\leq C
\end{equation}
for a sufficiently small $\varkappa$ and for all $(y,z)\in V_{0,\varkappa} \setminus S$.
\end{lemma}
\begin{proof}
Recall that the density of $\mu$ has the  expression $f(y,z)= g(y,z)\psi(y)z^{-\gamma-1}$. The proof of the lemma is based on relating the invariant measure to the occupation times of the process $X_t$. Namely, let $0 < \varkappa_1 < \varkappa_2$. Consider excursions by the process $X_t$ between $\Gamma_{\varkappa_1}$ and $\Gamma_{\varkappa_2}$. By observing the position of $X_t$ at the times of its successive returns to $\Gamma_{\varkappa_1}$ (after visiting $\Gamma_{\varkappa_2}$), we obtain a Markov chain on $\Gamma_{\varkappa_1}$. Similarly, by  
observing the position of $X_t$ at the times of  its successive returns to $\Gamma_{\varkappa_2}$ (after visiting $\Gamma_{\varkappa_1}$), we obtain a Markov chain on $\Gamma_{\varkappa_2}$. Let $\nu_1$ and $\nu_2$ be the invariant measures for these Markov chains.

For the process $X_t$ starting at $x \in \Gamma_{\varkappa_1}$,
let $T_1$ be the occupation time for $V_{\frac{\varkappa_1}{2}, \varkappa_1}$ during one excursion, i.e, 
\[
T_1 = \int_0^{\tau_1}
\chi_{V_{\frac{\varkappa_1}{2}, \varkappa_1}}(X_t) dt,
\]
where $\tau_1$ is the first time when the process returns to $\Gamma_{\varkappa_1}$ after visiting $\Gamma_{\varkappa_2}$. Similarly, for the processes starting at $\Gamma_{\varkappa_2}$, let $T_2$ be the occupation time for $V_{\varkappa_2, 2 \varkappa_2}$ during one excursion. By the Hasminskii formula (see \cite{khas1960ergodic}), 
\begin{equation} \label{meas1a} \mu(V_{\frac{\varkappa_1}{2}, \varkappa_1})
 =
\frac{\int_{\Gamma_{\varkappa_1}}\mathrm{E}_x T_1 d \nu_1(x)}{\int_{\Gamma_{\varkappa_2}}\mathrm{E}_x T_2 d \nu_2(x)} 
\mu(V_{\varkappa_2, 2\varkappa_2}).
\end{equation}
For the left-hand side, by the definition of $g$, 
$$
c \varkappa_1^{-\gamma}\inf_{V_{\frac{\varkappa_1}{2},\varkappa_1}} g  \leq 
\mu(V_{\frac{\varkappa_1}{2}, \varkappa_1}) \leq c \varkappa_1^{-\gamma} \sup_{V_{\frac{\varkappa_1}{2},\varkappa_1}} g ,
$$
where $c = \int_{V_{\frac{1}{2}, 1}} z^{-\gamma -1} \psi(y) d \pi(y) dz.$ 
Since $g$ is a non-negative and smooth function satisfying $\overline{L} g = 0$, by $(iii)$ of Lemma \ref{ellipticity estimation} and the last Remark in Section \ref{procbeh}, there exists $C >0$ such that 
\[
\sup_{V_{\frac{\varkappa_1}{2},\varkappa_1}} g \leq C \inf_{V_{\frac{\varkappa_1}{2},\varkappa_1}} g
\]
for all sufficiently small $\varkappa_1$.
Therefore, there exist $k_1,k_2>0$ such that
\begin{equation}
\label{relation of mu and g}
k_1 \varkappa_1^{-\gamma}\sup_{V_{\frac{\varkappa_1}{2},\varkappa_1}} g  \leq 
\mu(V_{\frac{\varkappa_1}{2}, \varkappa_1}) \leq k_2 \varkappa_1^{-\gamma} \inf_{V_{\frac{\varkappa_1}{2},\varkappa_1}} g.
\end{equation}
Thus, the upper and lower bounds on $g$ will follow if we provide the appropriate bounds on the right-hand side of~(\ref{meas1a}) for a fixed $\varkappa_2$ and for all sufficiently small $\varkappa_1$. Let us clarify the selection of $\varkappa_2$: $\varkappa_2$ is sufficiently small such that 

-the estimate on the transition probability with the boundary $\Gamma_{\varkappa_1}\bigcup \Gamma_{2\varkappa_2}$ in Lemma \ref{unperturbed transition probability} is valid with some small but fixed $\eta$, for all $\varkappa_1$ sufficiently small. More precisely, $\varkappa_2$ is set so that the estimates of the transition probabilities (\ref{excursion transition probability}) and (\ref{excursion transition small}) (below) hold.

-$(ii)$ in Lemma \ref{ellipticity estimation} holds for $K=\Gamma_{\frac{3}{2}\varkappa_2},\, U=\mathrm{ int} V_{\varkappa_2,2\varkappa_2}$.

In the following proof, $\varkappa_1$ is taken sufficiently small according to a fixed $\varkappa_2$ satisfying the above properties.

Since $\mu(V_{ {\varkappa_2}, 2\varkappa_2})$ is fixed, to provide an upper bound on the ratio of the integrals in 
(\ref{meas1a}), it is sufficient to estimate 
$\inf_{x\in \Gamma_{\varkappa_2}}\mathrm{E}_x T_2$ from below and $\sup_{x\in \Gamma_{\varkappa_1}} \mathrm{E}_x T_1$ from above. (the lower bound can be obtained similarly).

Let us consider a part of the occupation time to get the lower bound on $\inf_{x\in \Gamma_{\varkappa_2}}\mathrm{E}_x T_2$. Namely, by Lemma \ref{unperturbed transition probability} and the proper selection of $\varkappa_2$, there exists $C_1 >0$ such that
\begin{equation}
\label{excursion transition probability}
\mathrm{P}_x(X_{\tau(\Gamma_{\varkappa_1}\bigcup\Gamma_{\frac{3\varkappa_2}{2}})}\in \Gamma_{\varkappa_1}) \leq C_1 \varkappa_1^{-\gamma},~~x\in \Gamma_{\varkappa_2},
\end{equation}
for all $\varkappa_1$ sufficiently small. By $(ii)$ in Lemma \ref{ellipticity estimation} and the proper selection of $\varkappa_2$, there exists $C_2>0$ such that
$$
\mathrm{E}_x \tau(\Gamma_{\varkappa_2}\bigcup \Gamma_{2\varkappa_2}) \geq C_2,\, x\in \Gamma_{\frac{3\varkappa_2}{2}}.
$$
Therefore, considering a possible excursion from $x\in \Gamma_{\varkappa_2}$ to $\Gamma_{\frac{3 \varkappa_2}{2}}$ and back to $\Gamma_{\varkappa_2}$, using the strong Markov property, we get
\[
\inf_{x\in \Gamma_{\varkappa_2}}\mathrm{E}_x T_2 \geq (1- C_1 \varkappa_1^{-\gamma}) ( C_2 + \inf_{x\in \Gamma_{\varkappa_2}}\mathrm{E}_x T_2).
\]
Thus, there exists $C_3>0$ such that
$$
\inf_{x\in \Gamma_{\varkappa_2}} \mathrm{E}_x T_2 \geq C_3 \varkappa_1^\gamma
$$
for all $\varkappa_1$ sufficiently small. For the upper bound on $\sup_{x\in\Gamma_{\varkappa_1}} \mathrm{E}_xT_1$, 
similarly, by Lemma~\ref{unperturbed transition probability} and the proper selection of $\varkappa_2$, there exists $C_4>0$ such that 
\begin{equation}
\label{excursion transition small}
\mathrm{P}_x(X_{\tau(\Gamma_{\varkappa_1}\bigcup\Gamma_{\varkappa_2})}\in \Gamma_{\varkappa_2}) \geq C_4,~~x\in \Gamma_{2\varkappa_1},
\end{equation}
for all $\varkappa_1$ sufficiently small. Observe that $\tau(\Gamma_{2\varkappa_1}),\,x\in \Gamma_{\varkappa_1}$ is larger than the  occupation time of $V_{\frac{\varkappa_1}{2},\varkappa_1}$ during the time interval $[0,\tau(\Gamma_{2\varkappa_1)}]$. By Lemma \ref{unperturbed expectation exit time}, there exists $C_5 >0$ such that
$$
\max_{x\in \Gamma_{\varkappa_1}}\mathrm{E}_x \tau(\Gamma_{2\varkappa_1}) \leq C_5
$$
for all $\varkappa_1$ sufficiently small. Therefore, considering the transition from $x\in \Gamma_{\varkappa_1}$ to $\Gamma_{2 \varkappa_1}$ and, possibly back, using the strong Markov property, we obtain
\[
\sup_{x\in \Gamma_{\varkappa_1}}\mathrm{E}_x T_1 \leq C_5 + (1- C_4)  \sup_{x\in \Gamma_{\varkappa_1}}\mathrm{E}_x T_1.
\]
Therefore, there exists $C_6>0$ such that
\[
\sup_{x\in \Gamma_{\varkappa_1}}\mathrm{E}_x T_1 \leq C_6
\]
for all $\varkappa_1$ sufficiently small. The upper bound on $g$ now follows from (\ref{meas1a}) and (\ref{relation of mu and g}). As we mentioned, the lower bound can be obtained similarly.
\end{proof}
The next Lemma states that $g$ is asymptotically constant on the layers $V_{\varkappa, r \varkappa}$, as $\varkappa \downarrow 0$, where $r \geq 1$ is an arbitrary fixed number.
\begin{lemma} 
\label{g constant on interval}
Suppose that $r \geq 1$ is an arbitrary fixed number. Then the function
$g$ satisfies the property
$$
\lim_{\varkappa\downarrow 0}\frac{\sup_{x \in V_{\varkappa, r \varkappa}} g(x)}{\inf_{x \in V_{\varkappa, r \varkappa}} g(x)} = 1.
$$
\end{lemma}
\begin{proof}
Recall that $g$ satisfies (\ref{eqnbsr}) and that $\overline{X}_t$ is the process with the generator $\overline{M} + \overline{R}$. Note that the generator of $\overline{X}_t$  has a similar form to the generator of $X_t$ (which is $M +R$). Following the definition of $\overline{M}$ in (\ref{defbarm}), we have
\[
\overline{M}(\psi^{-1} z^\gamma)
=\psi^{-1}z^{\gamma+1}M^*(z^{-1})=0.
\]
Therefore, by Lemma \ref{lemmaspectral}, the surface $S$ is repelling for the process $\overline{X}_t$ with $\overline{\gamma}=\gamma<0$.

For the rest of the proof, the main idea is
to view $g$ as a solution to an elliptic equation with the boundary data provided by $g$ itself, which we know to satisfy (\ref{boundsong}). 
The equation can be solved  with the aid of the Feynman-Kac formula, and the processes starting at two nearby points can be coupled, with high probability, using (i) in Lemma~\ref{ellipticity estimation}. Let us now list some notation and the rules to choose parameters, before proceeding with a detailed argument. 

For two sub-probability measures $\mu_1, \mu_2$ defined on the same metric space $M$, let $\rho(\mu_1, \mu_2)$ be the largest constant such that there is a positive measure $\nu$ with $\nu(M) = \rho(\mu_1, \mu_2)$ such that $\mu_1 - \nu$, $\mu_2 - \nu$ are positive measures.

Recall that $\overline{X}_t = (\overline{Y}_t, \overline{Z}_t)$ is the process generated by $\overline{M} + \overline{R}$. Similarly to the last Remark in Section \ref{procbeh}, for all $\varkappa$ sufficiently small, let $\hat{X}_t^{\varkappa}$ be the process obtained from $\overline{X}_t$ in $V_{re^{-2}\varkappa, r\varkappa}$ after changing the variables $z \rightarrow z/\varkappa$. Note that $\hat{X}_t^{\varkappa}$ is defined in the domain $V_{re^{-2}, r}$. Therefore, the stochastic differential equation for $\hat{X}_t^{\varkappa}$ has uniformly bounded and non-degenerate coefficients for all $\varkappa$ sufficiently small. Let $\hat{\mu}^{x,\varkappa}$ be the measure induced by $\hat{X}^{\varkappa}_{\tau(\Gamma_{re^{-2}}\bigcup\Gamma_{r})}$ on $\Gamma_{r}$ (this is a sub-probability measure). The process is assumed to start at $x$.

By $(i)$ in Lemma \ref{ellipticity estimation}, there exists $C>0$ such that
\begin{equation}
\label{coupled measure}
\rho(\hat{\mu}^{x_1,\varkappa},\hat{\mu}^{x_2,\varkappa})\geq C,\quad  x_1,x_2 \in \Gamma_{re^{-1}},
\end{equation}
for all $\varkappa$ sufficiently small. Let $\varkappa_0$ be such that (\ref{coupled measure}) holds uniformly for all $0<\varkappa<\varkappa_0$ with the same constant $C$. Note that $C\in (0,1)$. Consider $\delta_1\in (0,1)$, which will be specified later.

By Lemma~\ref {unperturbed transition probability}, there exists $N_1$ that depends on $\delta_1$ such that, for all $N\geq N_1$, we have 
\begin{equation} \label{goingritht} \mathrm{P}_x(\overline{X}_{\tau(\Gamma_{\varkappa e^{-N}}\bigcup \Gamma_{r \varkappa e^N})}\in \Gamma_{r \varkappa e^N})\geq 1 -\delta_1,\quad  x\in V_{\varkappa,r\varkappa e^N},
\end{equation}
for all $\varkappa$ sufficiently small. (The application of Lemma~\ref{unperturbed transition probability} is not completely straightforward, since the process at hand is now $\overline{X}_t$ rather than $X_t$, and therefore Lemma~\ref{unperturbed transition probability}  allows us to control the transition probabilities between the layers $\overline{\Gamma}_\varkappa$ for different $\varkappa$, rather than between the layers  $\Gamma_\varkappa$. However, we can use the fact that $\Gamma_\varkappa$ lies between $\overline{\Gamma}_{\varkappa/a}$ and $\overline{\Gamma}_{\varkappa a}$ for sufficiently large $a$ and for all $\varkappa > 0$.)

Recall that $g$ solves 
\[
\overline{L}g = \overline{M}g + \overline{R}g + \overline{T} g = 0
\]
in $V_{\varkappa e^{-N}, r \varkappa e^N}$, where $\overline{M}$ is homogeneous in $z$, $\overline{R}$ has small coefficients when $\varkappa$ is small, and $\overline{T}$ is a potential term, which is small when $\varkappa$ is small. Let $u = u^\varkappa$ solve 
\[
 \overline{M}u + \overline{R}u  = 0
\]
in $V_{\varkappa e^{-N}, r \varkappa e^N}$, with
$u = g$ on $\Gamma_{\varkappa e^{-N}}\bigcup \Gamma_{r \varkappa e^N}$. By the Feynman-Kac formula, 
\begin{equation}
\label{feynman-kac of u}
\frac{u(x_1)}{u(x_2)}= \frac{\mathrm{E}_{x_1}g(\overline{X}_{\tau(\Gamma_{\varkappa e^{-N}}\bigcup \Gamma_{r\varkappa e^{N}})})}{\mathrm{E}_{x_2}g(\overline{X}_{\tau(\Gamma_{\varkappa e^{-N}}\bigcup \Gamma_{r\varkappa e^{N}})})},\quad x_1,x_2\in V_{\varkappa,r\varkappa}.
\end{equation}
We would like to show that the expressions in the numerator and denominator in the right-hand side are close. To this end,
we will consider the measures induced by the stopped process $\overline{X}_{\tau(\Gamma_{\varkappa e^{-N}}\bigcup \Gamma_{r\varkappa e^{N}})}$ with the starting point $x \in \Gamma_{r\varkappa e^{N-k}}$ for $k = 1,...,N$. Using induction on $k$, we will show that such measures do not depend much on $x$ for $k = N$, i.e., for $x \in  \Gamma_{r\varkappa }$.

 For $x \in V_{\varkappa e^{-N}, r\varkappa e^{N-1}}$,  let $\mu^{x}$ be the measure induced by 
$\overline{X}_{\tau(\Gamma_{\varkappa e^{-N}}\bigcup \Gamma_{r\varkappa e^{N}})}$ on $\Gamma_{r\varkappa e^{N}}$ (this is a sub-probability measure). Let us consider $\varkappa$ sufficiently small such that $\varkappa e^N\leq \varkappa_0$. By (\ref{coupled measure}),  we have
\begin{equation}
    \label{common factor c}
    \rho(\mu^{x_1}, \mu^{x_2}) \geq C,\quad x_1, x_2 \in \Gamma_{r\varkappa e^{N-1}},
\end{equation}
for all $\varkappa$ sufficiently small. Observe that the same value of $C$ will work if we reduce the value of $N$.

Now consider $x_1, x_2 \in \Gamma_{r\varkappa e^{N-2}}$. The processes starting at these points may get coupled later, upon reaching $\Gamma_{r\varkappa e^{N-1}}$ and then proceed to $\Gamma_{r\varkappa e^{N}}$. The uncoupled portions of the measures induced on $\Gamma_{r\varkappa e^{N-1}}$  may get coupled upon reaching $\Gamma_{r\varkappa e^{N}}$. Thus,
\[
\rho(\mu^{x_1}, \mu^{x_2}) \geq C (1 - \delta_1) + (1 - C - \delta_1) \inf_{x'_1, x'_2 \in \Gamma_{r \varkappa e^{N-1}}} \rho(\mu^{x'_1}, \mu^{x'_2}),\quad x_1,x_2\in \Gamma_{r\varkappa e^{N-2}},
\]
holds for $\delta_1$ in (\ref{goingritht}) and $C$ in (\ref{common factor c}).

Inductively, let us consider $a_k = \inf_{x_1, x_2 \in \Gamma_{r \varkappa e^{N-k}}} \rho(\mu^{x_1}, \mu^{x_2})$. The previous discussion shows that $a_1 \geq C$ and $a_2 \geq  C (1 - \delta_1) + (1 - C - \delta_1)a_1$. Arguing as above, we get
\[
a_k \geq  C (1 - \delta_1) + (1 - C - \delta_1)a_{k-1},~~1 < k \leq N.
\]
Therefore, for each $\delta_2 > 0$, there exists $N_2$ that depends on $\delta_1,\delta_2$ such that for all $N\geq N_2$, we have
\[
a_N = \inf_{x_1, x_2 \in \Gamma_{r \varkappa }} \rho(\mu^{x_1}, \mu^{x_2})
\geq \frac{(1-\delta_1) C}{\delta_1 + C} - \delta_2
\]
for all sufficiently small $\varkappa$. 

Note that the right-hand side can be made arbitrarily close to $1$ by taking $\delta_1$ and $\delta_2$ sufficiently small. By (\ref{goingritht}), the process starting in $V_{\varkappa, r \varkappa}$ will go to $\Gamma_{r\varkappa}$ before going to $\Gamma_{\varkappa e^{-N}}$ with high probability. Therefore, using the strong Markov property, we conclude that the infimum in the last inequality can be taken over $x_1, x_2 \in V_{\varkappa, r \varkappa}$, i.e., for each $\delta_3>0$, there exists a sufficiently large $N$ such that
\begin{equation}
\label{common factor conclusion}
\inf_{x_1, x_2 \in V_{\varkappa, r \varkappa }} \rho(\mu^{x_1}, \mu^{x_2})
\geq  1- \delta_3
\end{equation}
for all $\varkappa$ sufficiently small. From (\ref{feynman-kac of u}), (\ref{common factor conclusion}) and Lemma \ref{uniform bound on g}, it follows that, for each $\delta > 0$, there exists a sufficiently large $N$ such that
\begin{equation} \label{ratioofu}
1 - \delta \leq \frac{u(x_1)}{u(x_2)} \leq 1 + \delta,~~~x_1, x_2 \in V_{\varkappa, r\varkappa}
\end{equation}
for all $\varkappa$ sufficiently small. What remains is to get a similar inequality with $g$ instead of $u$.
By the Feynman-Kac formula, for $x\in V_{\varkappa,r\varkappa}$,
\[
g(x) = \mathrm{E}_x [g(\overline{X}_{\tau{(\Gamma_{\varkappa e^{-N}}\cup \Gamma_{r\varkappa e^{N}})}})\exp(\int_0^{\tau_{(\Gamma_{\varkappa e^{-N}}\bigcup \Gamma_{r \varkappa e^{N}})}} \overline{T}( \overline{X}_t)dt)]=
\]
\[
u(x) + \mathrm{E}_x [g(\overline{X}_{\tau{(\Gamma_{\varkappa e^{-N}}\cup \Gamma_{r\varkappa e^{N}})}})(\exp(\int_0^{\tau_{(\Gamma_{\varkappa e^{-N}}\bigcup \Gamma_{r \varkappa e^{N}})}} \overline{T}( \overline{X}_t)dt) - 1)].
\]
Since $g$ is uniformly bounded from above and below, it is sufficient to estimate
$$
\mathrm{E}_x [\exp(\int_0^{\tau_{(\Gamma_{\varkappa e^{-N}}\bigcup \Gamma_{r \varkappa e^{N}})}} \overline{T}( \overline{X}_t)dt) - 1].
$$
Since $\overline{T}(y,z)\rightarrow 0$ uniformly in $y$ as $z\downarrow 0$, for each $\lambda>0$ and each $N$,
$$
\mathrm{E}_x \exp( \int_0^{ \tau_{(\Gamma_{\varkappa e^{-N}}\bigcup \Gamma_{r \varkappa e^{N}})}}|\overline{T}(\overline{X}_t)|dt)\leq \mathrm{E}_x \exp( \lambda \tau_{(\Gamma_{\varkappa e^{-N}}\bigcup \Gamma_{r \varkappa e^{N}})}),\quad x\in V_{\varkappa e^{-N},r\varkappa e^N}
$$
for all $\varkappa$ sufficiently small. Using the fact that $\mathrm{P}_x (\tau_{(\Gamma_{\varkappa e^{-N}}\bigcup \Gamma_{r \varkappa e^{N}})} < c^{-1} )  > c > 0$ for some $c > 0$ and all sufficiently small $\varkappa$, and the fact that $\lambda$ can be chosen arbitrarily small, we conclude, by the Markov property, that
$$
\sup_{x\in V_{\varkappa e^{-N},r\varkappa e^N}}\mathrm{E}_x \exp( \lambda \tau_{(\Gamma_{\varkappa e^{-N}}\bigcup \Gamma_{r \varkappa e^{N}})}) 
$$
can be made arbitrarily close to $1$. Therefore, we can replace $u$ by $g$ in (\ref{ratioofu}) to obtain that, for each $\delta > 0$,
$$
1 - \delta \leq \frac{g(x_1)}{g(x_2)} \leq 1 + \delta,\quad x_1, x_2 \in V_{\varkappa, r\varkappa},
$$
for all $\varkappa$ sufficiently small. This proves Lemma \ref{g constant on interval}.\end{proof}

Let us continue with the proof of Theorem~\ref{invariant measure}. By Lemma \ref{uniform bound on g} and Lemma \ref{g constant on interval}, for an arbitrarily fixed number $r\geq 1$, there exists a function $\delta_0(\varkappa)$ such that $\lim_{\varkappa\downarrow 0}\delta_0(\varkappa)= 0$ and the function $g$ satisfies
\begin{equation}
    \label{g small error constant}
    \sup_{x\in V_{\varkappa,r\varkappa}} g - \inf_{x\in V_{\varkappa, r\varkappa}} g\leq \delta_0(\varkappa);
\end{equation}
equivalently, this states that $g$ is nearly constant in $V_{\varkappa, r \varkappa}$ for all $\varkappa$ sufficiently small. To prove Theorem \ref{invariant measure}, we still need to show that the constant can be chosen independently of $\varkappa$. Since $f(y,z)= g(y,z)\psi(y)z^{-\gamma-1}$, the following lemma implies Theorem~\ref{invariant measure}.
\begin{lemma}
\label{g uniform constant}
 There exists a function $\delta(t)$ such that $\lim_{t\downarrow 0} \delta(t) = 0$ and, for all sufficiently small $t$, there exists $\varkappa>0$ (depending on $t$) such that
$$
(1-\delta(t))\inf_{y'\in S}g(y',t) \leq g (y,z) \leq (1+\delta(t))\sup_{y'\in S}g(y',t),\quad (y,z)\in \mathrm{ int}\, V_{0,\varkappa}.
$$
\end{lemma}
\begin{proof}
Following the discussion at the beginning of the proof of Lemma \ref{g constant on interval}, we have $\overline{M}(\psi^{-1}z^\gamma)=0$, where $\psi(y)$ is a positive and continuous function on $S$. We will use the idea that $F(\overline{X}_t):=\psi^{-1}(\overline{Y}_t)\overline{Z}_t^\gamma$ is almost a martingale, while its convex power
$F^{1/\gamma}$ and concave power $F^{1/2}$
are almost sub- and super-martingales, respectively.

More precisely, applying $\overline{M}$ to $\psi^{-1}z^\gamma$ with the powers $1/\gamma<0$ and
$1/2$, by assumption (c), we see that there is $c>0$ such that
\[
\overline{M}(\psi^{-1/\gamma}z)\geq 2cz,
\qquad
\overline{M}(\psi^{-1/2}z^{\gamma/2})
\leq-2c\psi^{-1/2}z^{\gamma/2}.
\]
Since the coefficients of $\overline{R}$
and the coefficient of the potential term  $\overline{T}$ have a higher order of $z$, for all $z$ sufficiently small, we have
\begin{equation}
\label{test function}
(\overline{M}+\overline{R})(\psi^{-1/\gamma}z)\geq c z,
\qquad
\overline{L}(\psi^{-1/2}z^{\gamma/2})
\leq-c\psi^{-1/2}z^{\gamma/2}.
\end{equation}
Let us consider the following function
\[
\delta(t):=
Ct\max_{y'\in S}\psi(y')^{-1/\gamma}.
\]
where $C>0$ is chosen sufficiently large. More precisely, $C>\frac{2}{c}\sup_{x\in S^\delta}|T(x)|$ where $T(x)$ is the coefficient of the potential term for the operator $\overline{L}$. For all sufficiently small $t$, we have
\[
\begin{aligned}
\overline{L}(1+\delta(t)-C\psi^{-1/\gamma}z)
&\leq0,\\
\overline{L}(1-\delta(t)+C\psi^{-1/\gamma}z)
&\geq0,
\end{aligned}
\qquad 0<z<t,
\] 
where both functions inside the parentheses are positive,
and their values at $z=t$ are respectively bounded below and above by $1$.

By Lemma~\ref{uniform bound on g},
$g$ is uniformly bounded. Since
$\psi(y)^{-1/2}z^{\gamma/2}\to\infty$ uniformly in $y$
as $z\downarrow0$, for all sufficiently small $t$, let us consider $\xi>0$ sufficiently small such that, for all sufficiently small $\eta>0$, we have
\[
\begin{aligned}
&(1-\delta(t)+C\psi(y)^{-1/\gamma}z)
 \inf_{y'\in S}g(y',t)
 -\xi\psi(y)^{-1/2}z^{\gamma/2}\\
&\qquad\leq g(y,z)\\
&\qquad\leq
 (1+\delta(t)-C\psi(y)^{-1/\gamma}z)
 \sup_{y'\in S}g(y',t)
 +\xi\psi(y)^{-1/2}z^{\gamma/2}
\end{aligned}
\]
on the boundary $S\times\{\eta\}\bigcup S\times \{t\}$; moreover, applying the operator $\overline{L}$, the lower expression is non-negative and the upper expression is
non-positive. By the maximum principle, since $\overline{L}g=0$ and (\ref{test function}) holds, these inequalities therefore hold
throughout the domain $S\times (\eta,t)$.
Letting $\eta\downarrow0$ and then $\xi\downarrow0$
yields
\[
(1-\delta(t))\inf_{y'\in S}g(y',t)
\leq g(y,z)\leq
(1+\delta(t))\sup_{y'\in S}g(y',t),
\qquad 0<z<t.
\]
Since $\delta(t)\to0$ as $t\downarrow0$, we can finish the lemma by choosing $\varkappa>0$ such that
$\mathrm{int}\,V_{0,\varkappa}\subset S\times(0,t)$.
\end{proof}
Note that the form of the asymptotic behavior of the invariant measure (up to a
positive multiplicative constant), depends only on the  coefficients of the operator $L$ in a neighborhood of the surface $S$ in all the lemmas in this section. Thus, we have a more general version of Theorem \ref{invariant measure}. First, we introduce the required notation.

Suppose that a domain $D$ has the boundary $\partial D = \{S_1^D,..., S_{m_D}^D\}$ with $S_i^D,\, 1\leq i\leq m_D$, which are $C^9$-smooth surfaces. For each surface $S_i^D$, there is a product measure $d\pi_i\times dz$ defined in $D$ in a neighborhood of $S_i^D$, as mentioned before Lemma \ref{lemmaspectral}, and the exponent $\gamma_i^D<0$  defined in Lemma \ref{lemmaspectral}. Correspondingly, $\psi_i$ are  defined for all surfaces $S_i^D,\, 1\leq i\leq m_D$, as in Lemma \ref{adjointelliptic}. 

\begin{theorem}
\label{invariant measure multi-surfaces} There exist constants $c_i > 0,\, 1\leq i \leq m_D$, such that the density $f$ (with respect to $d\pi_i\times dz$) of the invariant measure of the process with the generator $L$ has the following asymptotic behavior in the neighborhood of the surface $S_i^D$:
$$
\lim_{z\downarrow 0} \frac{f(y,z)}{\psi_i(y) z^{-\gamma_i^D-1}} = c_i,\quad y\in S_i^D,\, 1\leq i \leq m_D.
$$
\end{theorem}
\section{Transition probabilities and exit times when exiting one domain} 

\label{Probability and expectation time for transition in one domain}
In this section, we focus on the behavior of the process $X_t^\varepsilon$ in one domain, so the dependence on $D$ is dropped from the notation. We consider a domain $D$ with the boundary $\partial D = \{S_1,...,S_{m}\}$ and let $M:=\{1,..., m\}$. Without loss of generality, let us assume that $0>\gamma_1\geq \gamma_2\geq...\geq \gamma_m$, and let $l$ be such that $\gamma_1=...=\gamma_{l} > \gamma_{l+1}$ if $\gamma_1>\gamma_m$ and $l = m$ otherwise. Let $V_{\varkappa_1,\varkappa_2}^i,\Gamma_\varkappa^i$ be defined as in (\ref{defineition of layers}) in a neighborhood of $S_i,\, i\in M$. We define $\Gamma_{\varkappa}:=\bigcup_{j\in M} \Gamma_{\varkappa}^j$ and $V_{\varkappa_1,\varkappa_2}:=\bigcup_{j\in M}V_{\varkappa_1,\varkappa_2}^j$ for $0\leq\varkappa_1<\varkappa_2$.

 The proofs of the statements in this section will be provided for the case when $D$ is bounded.   We assume that the starting point $X_0^\varepsilon\in K:=\mathrm{ cl} (D\setminus V_{0,\varkappa_0})$, where $\varkappa_0$ is chosen sufficiently small so that $X_t^\varepsilon$ cannot transit from $S_i$ to $S_j$ without going through $K$ for $i\neq j$ and so that the estimates in Section \ref{procbeh} hold with $\Gamma_{\varkappa_0}$ as the outer layer. Note that $K$ is a compact set when $D$ is bounded. For an unbounded domain $D$, we can define $K:=\mathrm{ cl} (D\setminus V_{0,\varkappa_0})\bigcap F$ for a sufficiently large closed ball $F$. Using Assumption (e), it is easy to show that all the results in this section hold with the starting point $X_0^\varepsilon\in K$ even if $D$ is unbounded.

\subsection{Transition probabilities for \texorpdfstring{$X_t^{\varepsilon}$}{TEXT}}
\label{transition probability in one domain}
 Let the transition probabilities to the surfaces be
$$
p_i^{x,\varepsilon} :=\mathrm{P}_x(X_{\tau^{\varepsilon}(\partial D)}^{\varepsilon} \in S_i),\quad i\in M,\, x\in K,
$$
where $\tau^{\varepsilon}(\partial D)$ is the first time the process $X_t^\varepsilon$ reaches $\partial D$.
\begin{theorem}
\label{proportion exit probability}
There exist positive constants $C_i,\rho_i$, ${i\in M}$, such that  
$$
p_i^{x,\varepsilon} \sim \frac{C_i\rho_i}{\sum_{j=1}^lC_j\rho_j}\varepsilon^{\gamma_1-\gamma_i}~~as~\varepsilon \downarrow 0,
$$
uniformly for $x \in K$. The numbers $C_i$ can be determined in terms of the coefficients of $L$ in $D$ and the numbers $\rho_i$ can be determined in terms of the coefficients of $L,\widetilde L$ in an arbitrarily small neighborhood of $S_i$.
\end{theorem}
\begin{proof}
Using the strong Markov property, we estimate $p_i^{x,\varepsilon}$ by studying the transition from $K$ to $\Gamma_{\varkappa}$ and the transition from $\Gamma_{\varkappa}$ to $\partial D$, where $0<\varkappa<\varkappa_0$ will be specified later.

Let us study the first part of the transition. For the unperturbed process $X_t$ starting in $K$, let us define a series of random times for excursions between $K$ and $\Gamma_{\varkappa}$:
$$\sigma_0^\varkappa = \tau_0^\varkappa = 0,\quad \sigma_n^\varkappa = \inf\{t>\tau_{n-1}^\varkappa| X_{t}\in \Gamma_{\varkappa}\},\, \tau_{n}^\varkappa = \inf\{t>\sigma_{n}^\varkappa|X_t\in \Gamma_{\varkappa_0}\},\quad n\geq 1.$$
Later, we will use the notation $\sigma_n^{\varkappa,\varepsilon}, \tau_n^{\varkappa,\varepsilon}$ for similar stopping times for the process $X^\varepsilon_t$. Observe that the process $X_t$ is non-degenerate in $\mathrm{ cl}(D \setminus V_{0, \varkappa_0/2})$. By $(i)$ in Lemma \ref{ellipticity estimation}, the measures induced by $X_{\tau_n^{\varkappa_0/2}}$ converge in total variation, exponentially fast in $n$, to a limiting measure $\nu$. (Here, the transitions are between $\Gamma_{\varkappa_0}$ and $\Gamma_{\varkappa_0/2}$ rather  than between $\Gamma_{\varkappa_0}$ and $\Gamma_\varkappa$ with generic $\varkappa.)$

By Lemma $4$, for each $n$, $\lim_{\varkappa\downarrow 0} \mathrm{P}_x(\tau_n^{\varkappa_0/2}<\tau(\Gamma_{\varkappa}))=1$. From the non-degeneracy, it also follows that there is a constant $C$ such that
\[
\sup_{x \in K} \mathrm{P}_x (\tau(\Gamma_\varkappa^i)<\tau_1^{\varkappa_0/2}) \leq C \inf_{x \in K} \mathrm{P}_x (\tau(\Gamma_\varkappa^i)<\tau_1^{\varkappa_0/2}),\quad i\in M.
\]
From these properties (exponential convergence of measures, the probability of early escape to $\Gamma_\varkappa$ going to zero, and a uniform  estimate on the probability of escape in one excursion), it easily follows that 
\begin{equation}
\label{mixing measure on Gamma kappa}
\lim_{\varkappa\downarrow 0} \frac{|\mathrm{P}_x(X_{\sigma_1^\varkappa}\in\Gamma_\varkappa^i)-\mathrm{P}_\nu(X_{\sigma_1^\varkappa}\in\Gamma_\varkappa^i)|}{\mathrm{P}_\nu(X_{\sigma_1^\varkappa}\in\Gamma_\varkappa^i)} =0,
\end{equation}
uniformly for $x\in K, i\in M$. 

Let $q^{x,\varkappa}_i = \mathrm{P}_x(X_{\sigma_1^\varkappa}\in\Gamma_\varkappa^i)$, $x \in K$. We just proved that the asymptotics of $q^{x,\varkappa}_i,\, i\in M$, as $\varkappa \downarrow 0$,
do not depend on $x$. Consider the Markov chain $X_{\tau^\varkappa_n}$ on $K$, and denote its invariant measure by $\nu_\varkappa$. Starting on $\Gamma_{\varkappa_0}$ with the measure $\nu_\varkappa$, let the measure on $\Gamma_\varkappa$ induced by $X_{\tau(\Gamma_\varkappa)}$ be denoted by $\pi_\varkappa$, and let the normalized restriction of $\pi_\varkappa$ to $\Gamma^i_\varkappa$ be denoted by $\pi^i_\varkappa$. Note that $\pi_\varkappa(\Gamma_\varkappa^i) = \int_{K}q_i^{x,\varkappa}d\nu_\varkappa(x)$. And by definition, we have
$$ (\int_{K}q_i^{\cdot,\varkappa}d\nu_\varkappa)
d\pi_\varkappa^i(x)   = d\pi_\varkappa(x),\quad x\in\Gamma_\varkappa^i.
$$
Let $T^i_\varkappa(x)$ be the occupation time in $V^i_{\varkappa/2, 
\varkappa}$ before hitting $K$ for the process $X_t$ starting at $x$. Observe that the probability to reach the set $V^i_{\varkappa/2, \varkappa}$ times the expected occupation time divided by the average excursion time is equal to the invariant measure of the set (Khasminskii's formula, see Theorem 2.1 in \cite{khas1960ergodic}), i.e., 
\[
\frac{\int_{\Gamma_\varkappa} \mathrm{E}_x T_{\varkappa}^i(x)d\pi_\varkappa(x)}{ \mathrm{E}_{\nu_\varkappa} \tau^\varkappa_1} =  \mu(V^i_{\varkappa/2, \varkappa}),\quad i \in M.
\]
Let us rewrite this in a slightly more convenient way:
\[
{ \int_K q^{x,\varkappa}_i d \nu_\varkappa(x)     =  \mathrm{E}_{\nu_\varkappa} \tau^\varkappa_1 \cdot   \mu(V^i_{\varkappa/2, \varkappa}) \cdot (\int_{\Gamma^i_\varkappa} \mathrm{E}_xT^i_\varkappa(x) d \pi_\varkappa^i(x)})^{-1}. 
\]
The first factor on the right-hand  side does not depend on $i$ (and that is the only property that we need of it). 
We have obtained the asymptotics for the second factor on the right-hand side in Theorem \ref{invariant measure multi-surfaces}. Namely, there exist positive constants $\bar c_i$ such that
$$
\lim_{\varkappa\downarrow 0}\frac{\mu(V^i_{\varkappa/2, \varkappa})}{\varkappa^{-\gamma_i}} = \bar c_i,\quad i\in M.
$$
Recall that we obtained upper and lower bounds on $\mathrm{E}_x T^i_\varkappa(x),\, x\in \Gamma_\varkappa^i$, in the proof of Lemma \ref{uniform bound on g}. Therefore, we can find a subsequence $\varkappa_k \downarrow 0$ such that
\[
\lim_{k \rightarrow \infty} 
(\int_{\Gamma^i_{\varkappa_k}} \mathrm{E}_x T^i_{\varkappa_k}(x) d \pi_{\varkappa_k}^i(x))^{-1} = c_i
\]
for each $i$. Combining these asymptotic results with the fact that
\[
\sum_{j\in M} \int_K q^{x,\varkappa}_j d \nu_\varkappa(x) =  1,
\]
for each $\varkappa$, we obtain
\[
\int_K q^{x,\varkappa_k}_i d \nu_{\varkappa_k}(x) \sim \frac{C_i \varkappa^{-\gamma_i}_k}{\sum_{j\in M} C_j \varkappa^{-\gamma_j}_k},\quad i\in M,
\]
as $k \rightarrow \infty$, where $C_i = c_i\bar c_i,\, i\in M$. Since the asymptotics of $q^{x,\varkappa_k}_i,\, i\in M$
do not depend on $x$, we obtain
\begin{equation}
\label{asymptotic probability to kappa}
q^{x,\varkappa_k}_i \sim \frac{C_i \varkappa^{-\gamma_i}_k}{\sum_{j\in M}C_j \varkappa^{-\gamma_j}_k},\quad i\in M,
\end{equation}
uniformly in $x \in K$ as $k\rightarrow \infty$. Note that the constants $C_i$ are  determined by the coefficients of $L$ in $D$. Note that the processes $X_t$ and $X_t^\varepsilon$ are close in finite time intervals for all $\varepsilon$ sufficiently small. Thus, for each $\varkappa>0$, since $X_t$ is non-degenerate in $\mathrm{ cl} (D\setminus V_{0,\varkappa})$, for each $\delta>0$, we have
$$
|\mathrm{P}_x(X_{\sigma_1^{\varkappa,\varepsilon}}^\varepsilon\in \Gamma_{\varkappa}^i) - q_i^{x,\varkappa}|/ q_i^{x,\varkappa} <\delta,\quad x\in K,\, i\in M,
$$
for all $\varepsilon$ sufficiently small. From (\ref{asymptotic probability to kappa}), for each $\delta>0$, we have
\begin{equation}
\label{first part transition}
|\mathrm{P}_x(X_{\sigma_1^{\varkappa_k,\varepsilon}}^\varepsilon\in \Gamma_{\varkappa_k}^i) - \frac{C_i \varkappa^{-\gamma_i}_k}{\sum_{j\in M}C_j \varkappa^{-\gamma_j}_k}|/\frac{C_i \varkappa^{-\gamma_i}_k}{\sum_{j\in M}C_j \varkappa^{-\gamma_j}_k}<\delta,\quad i\in M,
\end{equation}
uniformly in $x\in K$, for all $k$ sufficiently large, for all $\varepsilon$ sufficiently small (dependent of $\varkappa_k$).

Let us study the transition from $\Gamma_\varkappa$ to $\partial D$. The transition probability $\mathrm{P}_x(X_{\tau(S_i\bigcup \Gamma_{\varkappa_0}^i)}^\varepsilon \in S_i),\, x\in\Gamma_\varkappa^i$, has been estimated in Lemma \ref{transition probability to surface} with an error $\eta$ that can be made arbitrarily small by choosing small $\varkappa_0$, then choosing small $\varkappa$, and then choosing small $\varepsilon$. In fact, we can show that $\eta$ can be made arbitrarily small, while a fixed value of $\varkappa_0$ is considered, $\varkappa$ is taken sufficiently small (depending on $\eta$) and $\varepsilon$ is sufficiently small (depending on $\eta$, $\varkappa$). 
This generalization is easy to obtain from the fact that the process that reaches $\Gamma^i_{\varkappa_1}$ will proceed to $\Gamma^i_{\varkappa_2}$ before reaching $\Gamma_\varkappa^i$ with probability that can be arbitrarily close to one by taking $\varkappa$ sufficiently small, $\varepsilon$ sufficiently small (depending on $\varkappa$), for any $0 < \varkappa_1 < \varkappa_2$. Thus, for each $\eta>0$, there exists $\varkappa(\eta)$ such that
\begin{equation}
\label{second part transition}
|\mathrm{P}_x(X_{\tau(S_i\bigcup \Gamma_{\varkappa_0}^i)}^\varepsilon \in S_i) -\rho_i \varepsilon^{-\gamma_i} \varkappa^{\gamma_i}|/\rho_i \varepsilon^{-\gamma_i} \varkappa^{\gamma_i} \leq \eta,\quad x\in\Gamma_\varkappa^i,
\end{equation}
for all $0<\varkappa<\varkappa(\eta)$, for all $\varepsilon$ sufficiently small (depending on $\varkappa,\eta$). The constants $\rho_i$ can be determined in terms of the coefficients of $L,\widetilde L$ in an arbitrarily small neighborhood of $S_i$ (see Theorem 3.11 in \cite{freidlin2023perturbations}).

Now, let us put the estimates on the probabilities of the two stages of the transition together. Note that $p_i^{x,\varepsilon}$ satisfies
\begin{equation}
\label{proportion transition partition}
p_i^{x,\varepsilon}=\sum_{n=1}^\infty \mathrm{P}_x(\tau_{n-1}^{\varkappa,\varepsilon}<\tau^\varepsilon(\partial D))\cdot \mathrm{P}_x(X_{\sigma_n^{\varkappa,\varepsilon}}^\varepsilon\in \Gamma_\varkappa^i, \tau^{\varepsilon}(\partial D)<\tau_n^{\varkappa,\varepsilon}|\tau_{n-1}^{\varkappa,\varepsilon}<\tau^\varepsilon(\partial D)),
\end{equation}
for $x\in K, i\in M$. From (\ref{first part transition}) and (\ref{second part transition}), for each $\delta$, the second factor satisfies
$$
|\mathrm{P}_x(X_{\sigma_n^{\varkappa_k,\varepsilon}}^\varepsilon\in \Gamma_{\varkappa_k}^i, \tau^{\varepsilon}(\partial D)<\tau_n^{\varkappa_k,\varepsilon}|\tau_{n-1}^{\varkappa_k,\varepsilon}<\tau^\varepsilon(\partial D)) - \frac{C_i \rho_i \varepsilon^{-\gamma_i}}{\sum_{j\in M}C_j {\varkappa_k}^{-\gamma_j}}|/\frac{C_i \rho_i \varepsilon^{-\gamma_i}}{\sum_{j\in M}C_j {\varkappa_k}^{-\gamma_j}}<\delta,
$$
uniformly in $n\geq 1, x\in K$, for all $k$ sufficiently large, for all $\varepsilon$ sufficiently small (depending on $\varkappa_k$). Observe that $\sum_{j\in M}p_j^{x,\varepsilon} =1$ and the first factor in each term of the sum in (\ref{proportion transition partition}) does not depend on $i$. From this, by considering sufficiently small $\varkappa_k$, it follows that
$$
p_i^{x,\varepsilon} \sim \frac{C_i\rho_i \varepsilon^{-\gamma_i}}{\sum_{j=1}^l C_j \rho_j \varepsilon^{-\gamma_j}},\quad x\in K,\, i\in M,\, \mathrm{ as}\,\, \varepsilon\downarrow 0,
$$
which completes the proof of the theorem.
\end{proof}

\noindent
\textbf{Remark.} The numbers $\rho_i$, which appear in (\ref{second part transition}), are the same on either side of the surfaces $S_i$. This follows from the fact that $\rho_i$ depends only on the leading terms of the coefficients of $L$ and $\tilde{L}$ near $S_i$ (see the proof of Lemma 3.8 in \cite{freidlin2023perturbations}).
\subsection{Exit time of \texorpdfstring{$X_t^{\varepsilon}$}{TEXT} to surfaces}
\label{exit time through domains}
The arguments concerning the exit time will rely on the following lemma. The lemma looks very formal but has a very simple meaning, which we explain here. Consider repeatedly playing rounds of a game with very little chance of success in each round. The probability of success is controlled by a parameter $\varepsilon$. In each round, we first roll an asymmetric $m$-sided die to determine the type of prize that could potentially be won in that round. The probability that the die shows face $i$ is $q_i(\varepsilon)$, for $i \in \{1,\dots,m\}$. If the $i$-th side is showing, we toss a highly biased coin that lands heads up with probability $p_i(\varepsilon)$. If the coin lands heads up, the game stops and the type-$i$ prize is awarded; otherwise, we proceed to the next round. We assume that the overall probability of success in each round is tiny compared to the probability of seeing any fixed face $i$. Formally, $\lim_{\varepsilon\downarrow 0}\sum_{j=1}^m q_j(\varepsilon)p_j(\varepsilon)/q_i(\varepsilon)=0$, for all $i\in \{1,...,m\}$. The lemma implies that the number of rounds needed to win the prize, conditioned on the type of prize won, is asymptotically exponentially distributed. Moreover, the asymptotics of the parameter in the exponential distribution does not depend on the type of the prize.  

The lemma is somewhat more general than the above description. Namely, we allow the probabilities of winning each type of prize and of success to depend slightly on the state of a system that evolves in a Markovian way. Moreover, each round is allowed to last for a random time that  also depends on $\varepsilon$ and the type of prize corresponding to this round. This random time is what necessitates the two-tier (random prize and random success) formulation of the game; otherwise, we could simply consider a game with small probabilities of winning one of the prizes. In the lemma, the time period consists of two components - before the type of possible prize is announced and between the announcement of the type of prize and the decision on whether we win this round or not. 
\begin{lemma}
\label{abstract game}
For $\varepsilon>0$, let $I^{\varepsilon}_n, J^{\varepsilon}_n, s^{\varepsilon}_n, t^{\varepsilon}_n,\,\, n\in \mathbb{Z}_+$, be four sequences of random variables adapted to a filtration $\mathcal{F}_n^{\varepsilon}$. Suppose that there is another filtration $\mathcal{G}_n^\varepsilon$ such that $\mathcal{F}_n^{\varepsilon}\subseteq \mathcal{G}_n^\varepsilon\subseteq \mathcal{F}_{n+1}^{\varepsilon}$ and such that $J_{n+1}^\varepsilon, s_{n+1}^\varepsilon$ are $\mathcal{G}_n^\varepsilon$-measurable. Suppose that $I^{\varepsilon}_n$ takes values in $\{0,1\}$, $I^{\varepsilon}_0=1$, and that $I^{\varepsilon}_k=0$ implies that $I^{\varepsilon}_n=0$ for $n\geq k$. Suppose that $J^{\varepsilon}_n$ takes values in $M:=\{1,...,m\}$. Suppose that $s^{\varepsilon}_n, t^{\varepsilon}_n$ take values in $(0,\infty)$. All conditional asymptotic relations and conditional moment bounds below are required on $\{I_n^\varepsilon=1\}$.
In particular, the relations stated on
$\{J_{n+1}^\varepsilon=i\}$ are understood on
$\{I_n^\varepsilon=1,\ J_{n+1}^\varepsilon=i\}$. Suppose that there are positive functions $q_i(\varepsilon), \xi(\varepsilon), p_i(\varepsilon)$ such that
$$
\begin{aligned}
&\sum_{j\in M}q_j(\varepsilon) =1,\quad
\lim_{\varepsilon\downarrow 0}\mathrm{P}(J_{n+1}^\varepsilon = i|\mathcal{F}_{n}^{\varepsilon})/q_i(\varepsilon) = 1,\quad \lim_{\varepsilon\downarrow 0}\mathrm{E}(s^{\varepsilon}_{n+1}|\mathcal{F}_{n}^{\varepsilon})/\xi(\varepsilon) = 1,\\
&\lim_{\varepsilon\downarrow 0} \mathrm{E}(t_{n+1}^\varepsilon|\mathcal{F}_n^{\varepsilon})/\xi(\varepsilon) =0,\quad \lim_{\varepsilon\downarrow 0} \mathrm{P}(I_{n+1}^\varepsilon= 0|\mathcal{G}_{n}^\varepsilon)/p_i(\varepsilon) = 1\,\, \text{on}\,\,\{J_{n+1}^\varepsilon = i\},
\end{aligned}
$$
uniformly in $n\in \mathbb{Z}_+, i\in M$, and, in each of the formulas here and below, a version of the conditional expectation (which is defined up to a set of probability zero) can be taken such that the convergence is uniform with respect to the elements of the probability space. Let $p(\varepsilon) := \sum_{j\in M}q_j(\varepsilon) p_j(\varepsilon)$. In addition, we assume that
$$
\lim_{\varepsilon\downarrow 0} p(\varepsilon) /q_i(\varepsilon) = 0,\quad\lim_{\varepsilon\downarrow 0}\mathrm{E}(t_{n+1}^\varepsilon\chi_{\{I_{n+1}^\varepsilon=0\}}(\cdot)|\mathcal{G}_n^{\varepsilon})\frac{p(\varepsilon)}{p_i(\varepsilon)\xi(\varepsilon)}=0\,\, \text{on}\,\,\{J_{n+1}^\varepsilon = i\},
$$
uniformly in $n\in \mathbb{Z}_+, i\in M$. Let $N^{\varepsilon}:=\inf\{n: I^{\varepsilon}_n=0\}$. Let $E_i:=\{J^{\varepsilon}_{N^{\varepsilon}}=i\}$ and $T_i^{\varepsilon}$ be the random sum $\sum_{n=1}^{N^\varepsilon} (s^{\varepsilon}_{n} +t^{\varepsilon}_n)$ conditioned on $E_i$. We have the following results. 
\\
(i) For each $i\in M$,
$$ 
\lim_{\varepsilon\downarrow 0}\mathrm{E}T_i^{\varepsilon} \cdot p(\varepsilon)/\xi(\varepsilon) = 1.
$$
(ii) Assume that there is $C>0$ such that
$$
\mathrm{E}((s_{n+1}^\varepsilon)^2|\mathcal{F}_n^\varepsilon)\leq C(\xi(\varepsilon))^2,
$$
uniformly in $n\in\mathbb{Z}_+,\, i\in M$, for all $\varepsilon$ sufficiently small. Then we have, for each $t\in [0,\infty)$,
$$
\lim_{\varepsilon\downarrow 0}\mathrm{P}(\frac{p(\varepsilon)}{\xi(\varepsilon)}\cdot T_i^{\varepsilon}\geq t) = e^{-t}.
$$
\end{lemma}
\begin{proof}
For simplicity, we suppress the dependence on $\varepsilon$ in this proof. All conditional estimates below are uniform in $n\in\mathbb{Z}_+$ and are considered on $\{I_n=1\}$. Note that $p_i\leq p/q_i\to0,\, p\to0$. By summing over the possible end times, we have (uniformly in $n\in \mathbb{Z}_+$ and uniformly on the event $\{I_n=1\}$)
\[
\mathrm{P}(E_i|\mathcal F_n)\sim\frac{q_ip_i}{p},\quad \text{as} \,\, \varepsilon\downarrow 0.
\] Next, conditioning on another filtration $\mathcal{G}_n$, we have
\[
\begin{aligned}
\mathrm{P}(E_i|\mathcal G_n)
={}&
\chi_{\{J_{n+1}=i\}}
\mathrm{P}(I_{n+1}=0|\mathcal G_n)\\
&+
\mathrm{E}(
I_{n+1}\mathrm{P}(E_i|\mathcal F_{n+1})
|\mathcal G_n).
\end{aligned}
\]
Note that the first term is $o(q_ip_i/p)$, since
$p_i/(q_ip_i/p)=p/q_i\to0$; the second is asymptotically equivalent to $q_ip_i/p$, since
$\mathrm{P}(I_{n+1}=1|\mathcal G_n)\to1$ uniformly. Therefore, we have
\[
\mathrm{P}(E_i|\mathcal G_n)
\sim\mathrm{P}(E_i|\mathcal F_n)
\sim\frac{q_ip_i}{p}.
\]
Below, conditioning on both $\mathcal F_n$ and $E_i$ means
\[
\mathrm E(U\mid\mathcal F_n,E_i)
=
\frac{\mathrm E(U\chi_{E_i}\mid\mathcal F_n)}
{\mathrm P(E_i\mid\mathcal F_n)}
\]
on $\{I_n=1\}$. Since $s_{n+1}$ is
$\mathcal G_n$-measurable, the tower property
gives the first identity below. The second follows
from the fact that, on $\{I_n=1\}$,
$\{I_{n+1}=0\}\cap E_i
=\{I_{n+1}=0,J_{n+1}=i\}$.
Thus,
\[
\begin{aligned}
\mathrm{E}(s_{n+1}|\mathcal F_n,E_i)
&=
\frac{
\mathrm{E}[s_{n+1}\mathrm{P}(E_i|\mathcal G_n)
|\mathcal F_n]
}{
\mathrm{P}(E_i|\mathcal F_n)
}
\sim\xi,\\
\mathrm{P}(I_{n+1}=0|\mathcal F_n,E_i)
&=
\frac{
\mathrm{P}(I_{n+1}=0,J_{n+1}=i|\mathcal F_n)
}{
\mathrm{P}(E_i|\mathcal F_n)
}
\sim p.
\end{aligned}
\]
To estimate $t_{n+1}$, the last assumption and the
tower property imply
\[
\mathrm{E}(
t_{n+1}\chi_{\{I_{n+1}=0,J_{n+1}=i\}}
|\mathcal F_n)
=o\left(\frac{q_ip_i\xi}{p}\right).
\]
On $\{I_{n+1}=1\}$, we have
$\mathrm{P}(E_i|\mathcal F_{n+1})\sim q_ip_i/p$.
Splitting according to whether $I_{n+1}=0$, we obtain
\[
\begin{aligned}
&\mathrm E(t_{n+1}\mid\mathcal F_n,E_i)\\
&=\frac{1}{\mathrm P(E_i\mid\mathcal F_n)}
\Big[
\mathrm E\bigl(
t_{n+1}\chi_{\{I_{n+1}=0,J_{n+1}=i\}}
\mid\mathcal F_n\bigr)\\
&\qquad\qquad+
\mathrm E\bigl(
t_{n+1}I_{n+1}
\mathrm P(E_i\mid\mathcal F_{n+1})
\mid\mathcal F_n\bigr)
\Big].
\end{aligned}
\]
For the first term, conditioning on $\mathcal G_n$
and using the last assumption of the lemma gives
\[
\begin{aligned}
&\mathrm E\bigl(
t_{n+1}\chi_{\{I_{n+1}=0,J_{n+1}=i\}}
\mid\mathcal F_n\bigr)\\
&=
\mathrm E\bigl[
\chi_{\{J_{n+1}=i\}}
\mathrm E(t_{n+1}\chi_{\{I_{n+1}=0\}}
\mid\mathcal G_n)
\mid\mathcal F_n\bigr]\\
&=o\left(\frac{q_ip_i\xi}{p}\right).
\end{aligned}
\]
For the second term, since
$\mathrm P(E_i\mid\mathcal F_{n+1})
\sim q_ip_i/p$ on $\{I_{n+1}=1\}$,
we have
\[
\begin{aligned}
&\mathrm E\bigl(
t_{n+1}I_{n+1}
\mathrm P(E_i\mid\mathcal F_{n+1})
\mid\mathcal F_n\bigr)\le (1+o(1))\frac{q_ip_i}{p}
\mathrm E(t_{n+1}\mid\mathcal F_n).
\end{aligned}
\]
Combining these estimates with
$\mathrm P(E_i\mid\mathcal F_n)\sim q_ip_i/p$,
we conclude that
\[
\begin{aligned}
\mathrm E(t_{n+1}\mid\mathcal F_n,E_i)
&\le o(\xi)+(1+o(1))
\mathrm E(t_{n+1}\mid\mathcal F_n)=o(\xi).
\end{aligned}
\]
Note that $\mathrm{E}(N| E_i)\sim p^{-1}$.
Consequently, we have the formula, where the equality follows from the formula of total probability, and the asymptotics relations have been justified above:
\[
\mathrm{E}T_i=
\sum_{n=0}^{\infty}
\mathrm{E}(
I_n\mathrm{E}(s_{n+1}+t_{n+1}|\mathcal F_n,E_i)
| E_i)\sim \xi\,\mathrm{E}(N| E_i)
\sim\frac{\xi}{p}.
\]
For (ii), the same conditioning calculation applied to
$s_{n+1}^2$ gives
\[
\mathrm{E}(s_{n+1}^2|\mathcal F_n,E_i)
\leq2C\xi^2
\]
for sufficiently small $\varepsilon$.
Lemma~6.8 of \cite{koralov2024metastable} gives the convergence of $\sum_{n=1}^{N}s_n/\mathrm{E}T'_i$, conditioned on $E_i$, to the exponential distribution with parameter $1$, where $T'_i$ is defined as $T_i$ but with the terms $t_n$ omitted. Note that the sum of the $t_n$, multiplied by $p/\xi$, tends to zero in
 $L^1$ conditioning on $E_i$, by the preceding estimate.
Adding it therefore preserves the limiting distribution,
which proves (ii).
\end{proof}

In the next two lemmas, we will study $T_i^\varepsilon$, the exit time from a domain $D$ conditioned on $X_t^\varepsilon$ exiting through a particular surface $S_i$. More precisely, let $E_i : = \left\{ X_{\tau^\varepsilon(\partial D)}^{\varepsilon} \in S_i \right\},\, i \in M$, and $T^{\varepsilon}_i$ be the exit time $\tau^\varepsilon(\partial D)$ conditioned on $E_i$. We will show that, if the process starts in $K$,  $T_i^\varepsilon$ satisfies the assumptions of Lemma \ref{abstract game}, and thus the results in Lemma \ref{abstract game} hold for $T_i^\varepsilon$.
\begin{lemma}
\label{expectation exit time}
There is a positive constant $C$ such that
$$
\mathrm{E}_x T_i^\varepsilon \sim \frac{C}{\sum_{j=1}^lC_j \rho_j} \varepsilon^{\gamma_1},~~i \in M,\,as~\varepsilon \downarrow 0,
$$
uniformly in $x\in K$; moreover, $C$ is determined in terms of the coefficients of $L$ in $D$ and $C_i,\rho_i$ are the same as in Theorem \ref{proportion exit probability}.
\end{lemma}
\begin{proof}
For each excursion from $K$ to $\Gamma_\varkappa$ and back, we have $$
[\tau_{n-1}^{\varkappa,\varepsilon}, \tau_{n}^{\varkappa,\varepsilon}] = [\tau_{n-1}^{\varkappa,\varepsilon}, \sigma_n^{\varkappa,\varepsilon}]\cup (\sigma_n^{\varkappa,\varepsilon}, \tau_{n}^{\varkappa,\varepsilon}],\quad n\geq 1.
$$
We will estimate $\mathrm{E}_x T_i^\varepsilon$ using $(i)$ in Lemma \ref{abstract game}. More precisely, we will introduce a sequence $\varkappa\downarrow 0$   and consider all $0 < \varepsilon \leq \varepsilon_0(\varkappa)$ with $\varepsilon_0(\varkappa)$ going to zero sufficiently slowly as $\varkappa \downarrow 0$. For each pair $(\varkappa,\varepsilon)$, we estimate the total time that $X_t^\varepsilon$ spends during the excursions from $K$ to $\Gamma_\varkappa$ and back before reaching $\partial D$, conditioned on the event $E_i$. We will then take the joint limit in $\varkappa$ and $\varepsilon$, subject to $0<\varepsilon \leq \varepsilon_0(\varkappa)$, to obtain the asymptotics of $\mathrm{E}_x T_i^\varepsilon$ (note that this quantity does not depend on  $\varkappa$, but the introduction of $\varkappa$ is necessary for our analysis). Note that Lemma \ref{abstract game} is formulated using a single parameter $\varepsilon$, but it can be applied irrespective of the dimensionality of the parameter space.

Let us first clarify the definition of the quantities  $\mathcal{F}_n^{\varkappa,\varepsilon},\mathcal{G}_n^{\varkappa,\varepsilon},I_n^{\varkappa,\varepsilon},J_n^{\varkappa,\varepsilon},s_n^{\varkappa,\varepsilon},t_n^{\varkappa,\varepsilon}$ for this Lemma. Let $\mathcal{F}_n^{\varkappa,\varepsilon}$ be the filtration generated by $\{X_{t\wedge \tau_n^{\varkappa,\varepsilon}}^\varepsilon,\, t\geq 0\}$ and $\mathcal{G}_n^{\varkappa,\varepsilon}$ be the filtration generated by $\{X_{t\wedge \sigma_{n+1}^{\varkappa,\varepsilon}}^\varepsilon,\, t\geq 0\}$. Define
\[
\begin{aligned}
&I^{\varkappa,\varepsilon}_n = \chi_{\{\tau_{n}^{\varkappa,\varepsilon}<\tau^\varepsilon(\partial D)\}}(\cdot); ~~J^{\varkappa,\varepsilon}_n = j~~\mathrm{ whenever}~
X_{\sigma_n^{\varkappa,\varepsilon}}^\varepsilon\in \Gamma_{\varkappa}^j;\\
&s^{\varkappa,\varepsilon}_n = 
\sigma_n^{\varkappa,\varepsilon} - \tau_{n-1}^{\varkappa,\varepsilon};~~ t^{\varkappa,\varepsilon}_n = \tau^{\varkappa,\varepsilon}_n \wedge \tau^\varepsilon(\partial D)-\sigma^{\varkappa,\varepsilon}_n.\\
\end{aligned}
\]
Now, let us check that the assumptions for $(i)$ in Lemma \ref{abstract game} are satisfied. The estimates (\ref{first part transition}) and
(\ref{second part transition}), together with the
strong Markov property, provide us with the following bounds. Let $q_i(\varkappa,\varepsilon):= C_i\varkappa^{-\gamma_i}/(\sum_{j\in M}C_j \varkappa^{-\gamma_j})$ and $p_i(\varkappa,\varepsilon):= \rho_i \varepsilon^{-\gamma_i}\varkappa^{\gamma_i}$ for $i\in M$. By definition, we have $p(\varkappa,\varepsilon) = (\sum_{j\in M}C_j \rho_j \varepsilon^{-\gamma_j})/(\sum_{j\in M}C_j \varkappa^{-\gamma_j})$. For each $\delta>0$, we have
\begin{equation}
    \label{transition probability for J}
    |\mathrm{P}_x(J_{n+1}^{\varkappa,\varepsilon}= i|\mathcal{F}^{\varkappa,\varepsilon}_n)-q_i(\varkappa,\varepsilon)|/ q_i(\varkappa,\varepsilon)<\delta,\quad x\in K, n\geq 0, i\in M,
\end{equation}
\begin{equation}
    \label{transition probability for I}
    |\mathrm{P}_x(I_{n+1}^{\varkappa,\varepsilon}= 0|\mathcal{G}_n^{\varkappa,\varepsilon})-p_i(\varkappa,\varepsilon)|/ p_i(\varkappa,\varepsilon)<\delta\,\, \text{on}\,\,\{J_{n+1}^\varepsilon = i\},\quad x\in K, n\geq 0, i\in M,
\end{equation}
for all $\varkappa$ sufficiently small, for all $\varepsilon$ sufficiently small, on $\{I_n^{\varkappa,\varepsilon}=1\}$.

Now, let us  estimate $\mathrm{E}_x s_n^{\varkappa,\varepsilon}$. Since $X_t^\varepsilon$ is non-degenerate in $D\setminus V_{0,\varkappa_0/2}$ and has bounded coefficients in $D$, by considering the excursions from $K$ to $\Gamma_{\varkappa_0/2}$ and back, and by Lemma \ref{perturbed transition probability} and Lemma \ref{perturbed expectation time}, there is $c_1,c_2>0$ such that
\begin{equation}
\label{estimate expectation time from kappa0 to kappa}
c_1 \varkappa^{\gamma_1}\leq \inf_{x\in K}\mathrm{E}_x \tau^\varepsilon(\Gamma_\varkappa)\leq \sup_{x\in K}\mathrm{E}_x \tau^\varepsilon(\Gamma_\varkappa)\leq c_2\varkappa^{\gamma_1},
\end{equation}
for all $\varkappa$ sufficiently small, for all $\varepsilon$ sufficiently small (depending on $\varkappa$). Let $\xi(\varkappa):= \mathrm{E}_\nu(\tau(\Gamma_\varkappa))$, where the measure $\nu$ on $\Gamma_{\varkappa_0}$ is defined in the proof of Lemma \ref{proportion exit probability} and $\tau$ is generated by $X_t$. By (\ref{estimate expectation time from kappa0 to kappa}), we can find a subsequence of $\varkappa_k$ of the sequence used in (\ref{asymptotic probability to kappa}) (still denoted as $\varkappa_k$) such that $\varkappa_k\downarrow 0$ as $k\to \infty$ and $c_3>0$ such that for each $\delta>0$, we have
\begin{equation}
\label{precise expectation time from kappa0 to kappa}
|\xi(\varkappa_k)-c_3\varkappa_k^{\gamma_1}|/(c_3\varkappa_k^{\gamma_1}) <\delta,
\end{equation}
for all $k$ sufficiently large. Moreover, similarly to (\ref{mixing measure on Gamma kappa}) and the discussion above it, for each $\delta>0$, we have
\begin{equation}
\label{first-type intervals}
|
\frac{
\mathrm{E}_x(s_{n+1}^{\varkappa,\varepsilon}
|\mathcal{F}_n^{\varkappa,\varepsilon})
}{\xi(\varkappa)}-1|<\delta
\quad\text{on }\{I_n^{\varkappa,\varepsilon}=1\}.
\end{equation}
for all $\varkappa$ sufficiently small, for all $\varepsilon$ sufficiently small (depending on $\varkappa$).

By Lemma \ref{perturbed expectation time} and (\ref{estimate expectation time from kappa0 to kappa}), for each $\delta>0$, we have
\begin{equation}
\label{second type of interval}
\frac{
\mathrm{E}_x(t_{n+1}^{\varkappa,\varepsilon}
|\mathcal{F}_n^{\varkappa,\varepsilon})
}{\xi(\varkappa)}<\delta
\quad\text{on }\{I_n^{\varkappa,\varepsilon}=1\}.
\end{equation}
for all $\varkappa$ sufficiently small, for all $\varepsilon$ sufficiently small (depending on $\varkappa$).
Next, let us estimate the moments of the conditional random variable $t_{n+1}^{\varkappa,\varepsilon}\chi_{\{I_{n+1}^{\varkappa,\varepsilon}=0\}}(\cdot)|\mathcal{G}_n^{\varkappa,\varepsilon}$ on $\{I^{\varkappa,\varepsilon}_n=1,\,J_n^{\varkappa,\varepsilon}=i\}$. We treat the first and second moments together, since the latter will be used in Lemma~\ref{square of the exit time}. For $i\in M$, let us choose
$\hat\gamma:=\gamma_i-\gamma_1/2\in(\gamma_i,0)$. As in the proof of Lemma \ref{lemmaspectral} (see \cite{freidlin2023perturbations}), there are a positive
function $\hat v(y,z)=\hat\varphi(y)z^{\hat\gamma}$
and $\lambda>0$ such that
\[
L^\varepsilon\hat v\leq-\lambda\hat v
\quad\text{in }V_{R\varepsilon,\varkappa_0}^i,
\]
provided that $\varkappa_0$ is sufficiently small and the fixed constant $R$ is sufficiently large, for all $\varepsilon$ sufficiently small (depending on $\varkappa_0,R$). Let us suppress the dependence on $\varepsilon$, and consider the stopping times
\[
\sigma=\tau^\varepsilon
  (\Gamma_{R\varepsilon}^i\cup\Gamma_{\varkappa_0}^i),
\quad
\tau=\tau^\varepsilon
  (S_i\cup\Gamma_{\varkappa_0}^i).
\]
Using It\^o's formula up to $\sigma\wedge t$ and letting $t\to\infty$, there is $c_4>0$ such that
\begin{equation}
\label{exponentional moment ito}
\mathrm{E}_x\!\left[
e^{\lambda\sigma}
\chi_{\{X_\sigma^\varepsilon\in\Gamma_{R\varepsilon}^i\}}
\right]
\leq
\frac{\hat v(x)}
     {\inf_{\Gamma_{R\varepsilon}^i}\hat v}
\leq c_4(\varkappa/\varepsilon)^{\hat\gamma},
\qquad x\in\Gamma_\varkappa^i.
\end{equation}
Consequently, the same bound holds with
$e^{\lambda\sigma}$ replaced by $1+\sigma^r$,
for $r=1,2$. Note that by Lemma~\ref{perturbed expectation time}, there is $c_5>0$ such that
\[
\sup_{y\in V_{0,\varkappa_0}^i}\mathrm{E}_y\tau
\leq c_5(1+|\ln\varepsilon|).
\]
By the strong Markov property, there is $c_6>0$
\[
\mathrm{E}_y\tau^2
=2\mathrm{E}_y\int_0^\tau
  \mathrm{E}_{X_s^\varepsilon}\tau\,ds
\leq c_6(1+|\ln\varepsilon|)^2.
\]
Note that on the event that $S_i$ is reached before
$\Gamma_{\varkappa_0}^i$, $X^\varepsilon_t$ first reaches
$\Gamma_{R\varepsilon}^i$. Representing $[0,\tau]  = [0,\sigma) \cup [\sigma,\tau]$, noting that (\ref{exponentional moment ito}) provides an estimate on the probability of reaching $\Gamma^i_{R \varepsilon}$ before $\Gamma^i_{\varkappa_0}$, and  using the strong
Markov property, we see that there is $c_7>0$ such that
\begin{equation}
\label{last step exit time}
\sup_{x\in\Gamma_\varkappa^i}
\mathrm{E}_x[
(\tau^\varepsilon(S_i))^r
\chi_{\{\tau^\varepsilon(S_i)<
         \tau^\varepsilon(\Gamma_{\varkappa_0}^i)\}}]\leq
c_7(\varkappa/\varepsilon)^{\hat\gamma}
  (1+|\ln\varepsilon|)^r,\quad r=1,2.
\end{equation}
Note that for $\varkappa$ sufficiently small, we have
$p_i(\varkappa,\varepsilon)/
 p(\varkappa,\varepsilon)
=O_\varkappa(\varepsilon^{\gamma_1-\gamma_i})$.
Thus, after multiplication by $p/p_i$, the right-hand side of
(\ref{last step exit time}) becomes
\[
O_\varkappa\!\left(
\varepsilon^{-\gamma_1/2}
(1+|\ln\varepsilon|)^r
\right)\longrightarrow0,\quad \text{as}\,\varepsilon\downarrow 0.
\]
In particular, the case $r=1$ and the strong Markov
property at $\sigma_{n+1}^{\varkappa,\varepsilon}$ imply
\begin{equation}
\label{last step exit time assumption}
\begin{aligned}
&\mathrm{E}_x\!\left[
t_{n+1}^{\varkappa,\varepsilon}
\chi_{\{I_{n+1}^{\varkappa,\varepsilon}=0\}}
|\mathcal G_n^{\varkappa,\varepsilon}
\right]
\frac{p(\varkappa,\varepsilon)}
     {p_i(\varkappa,\varepsilon)\xi(\varkappa)}
\longrightarrow0,
\end{aligned}
\end{equation}
on
$\{I_n^{\varkappa,\varepsilon}=1,
J_{n+1}^{\varkappa,\varepsilon}=i\}$,
uniformly in $x\in K$, $n$, and $i$, as
$\varepsilon\downarrow0$ for each $\varkappa$ sufficiently small.

Now, we clarify the choice of $(\varkappa,\varepsilon)$ so that the terms on the left-hand sides of (\ref{transition probability for J}), (\ref{transition probability for I}), (\ref{precise expectation time from kappa0 to kappa}), (\ref{first-type intervals}), (\ref{second type of interval}), (\ref{last step exit time assumption}) have limit $0$ as $\varkappa\downarrow 0, \varepsilon\downarrow 0$. These limits will imply that the assumptions for $(i)$ in Lemma \ref{abstract game} are satisfied. By going to a further subsequence of $\varkappa_k$ (still denoted as $\varkappa_k$),  there is $\varepsilon_0(\varkappa_k)$ such that these estimates hold with the right-hand side ${1}/{k}$ with $\varkappa=\varkappa_k$ and for all $0<\varepsilon\leq\varepsilon_0(\varkappa_k)$. This implies that the assumptions in Lemma  \ref{abstract game} are satisfied and the conclusion of the lemma holds, i.e., 
\begin{equation}
\label{total expectation conditional exit time}
\lim_{\varkappa_k\downarrow 0,0<\varepsilon \leq \varepsilon_0(\varkappa_k)}|\mathrm{E}_xT_i^{\varepsilon} - \xi(\varkappa)\cdot \frac{\varepsilon^{\gamma_1}\sum_{j\in M}C_j {\varkappa}^{-\gamma_j}}{\sum_{j=1}^l C_j\rho_j}|/(\xi(\varkappa)\cdot \frac{\varepsilon^{\gamma_1}\sum_{j\in M}C_j {\varkappa}^{-\gamma_j}}{\sum_{j=1}^l C_j\rho_j})=0,
\end{equation}
uniformly in $x\in K,\, i\in M$. Note that from (\ref{precise expectation time from kappa0 to kappa}) we have
$$
\lim_{\varkappa_k\downarrow 0}\xi(\varkappa_k)\sum_{j\in M}C_j {\varkappa_k}^{-\gamma_j} = c_3\sum_{j=1}^l C_j=:C,
$$
where $C$ can be determined by the coefficients of $L$ in $D$. Thus, (\ref{total expectation conditional exit time}) implies that
$$
\lim_{\varkappa_k\downarrow 0,0<\varepsilon \leq \varepsilon_0(\varkappa_k)}|\mathrm{E}_xT_i^{\varepsilon} - \frac{C}{\sum_{j=1}^l C_j\rho_j}\varepsilon^{\gamma_1}|/(\frac{C}{\sum_{j=1}^l C_j\rho_j}\varepsilon^{\gamma_1})=0,
$$
uniformly in $x\in K,\, i\in M$. Observe that the quantity we are taking the limit of does not depend on $\varkappa$. This easily implies that 
$$
\lim_{\varepsilon\downarrow 0}|\mathrm{E}_xT_i^{\varepsilon} - \frac{C}{\sum_{j=1}^l C_j\rho_j}\varepsilon^{\gamma_1}|/(\frac{C}{\sum_{j=1}^l C_j\rho_j}\varepsilon^{\gamma_1})=0,
$$
uniformly in $x\in K,\, i\in M$, which completes the proof of the lemma.
\end{proof}
The following two lemmas are proved in a similar way to Lemma \ref{expectation exit time}. We therefore only sketch their proofs.
The distribution of $T^{\varepsilon}_i$ has the following property.
\begin{lemma}
\label{exponential exit time}
For each $t\in [0,\infty)$, $T_i^{\varepsilon}$ satisfies
$$
\lim_{\varepsilon\downarrow 0} \mathrm{P}_x((\frac{C}{\sum_{j=1}^l C_j\rho_j} \varepsilon^{\gamma_1})^{-1}T_i^\varepsilon\geq t) = e^{-t},
$$
uniformly in $x\in K,\, i\in M$.
\end{lemma}
\noindent
\begin{proof}
We will estimate the distribution of $T_i^\varepsilon$ using $(ii)$ in Lemma \ref{abstract game}. The quantities $\mathcal{F}_n^{\varkappa,\varepsilon}$, $\mathcal{G}_n^{\varkappa,\varepsilon}$, $I_n^{\varkappa,\varepsilon}$, $J_n^{\varkappa,\varepsilon}$, $s_n^{\varkappa,\varepsilon}$, $t_n^{\varkappa,\varepsilon}$ are defined in the same way as in Lemma \ref{expectation exit time}. Let us check the additional assumption for $s^{\varkappa,\varepsilon}_n=\sigma_n^{\varkappa,\varepsilon}-\tau_{n-1}^{\varkappa,\varepsilon}$ in $(ii)$ in Lemma \ref{abstract game}. To check this assumption, by the definition of $s_n^{\varkappa,\varepsilon}$ and $\xi(\varkappa)$, it is sufficient to show that there is $c_1>0$ such that
\begin{equation}
\label{variance bound}
\sup_{x\in K}\mathrm{E}_{x}(\tau^\varepsilon(\Gamma_\varkappa))^2\leq c_1 \varkappa^{2\gamma_1},
\end{equation}
for all $\varkappa$ sufficiently small, for all $\varepsilon$ sufficiently small (depending on $\varkappa$). By Lemma~\ref{perturbed expectation time} and (\ref{estimate expectation time from kappa0 to kappa}), we have
\[
\sup_{x\in D\setminus V_{0,\varkappa}}
\mathrm{E}_x\tau^\varepsilon(\Gamma_\varkappa)
\leq c_2(1+\ln(\varkappa_0/\varkappa)) +
\sup_{x\in K}\mathrm{E}_x\tau^\varepsilon(\Gamma_\varkappa)
\leq c_3\varkappa^{\gamma_1},
\]
for all $\varkappa$ sufficiently small, for all $\varepsilon$ sufficiently small. Indeed, the process starting outside $K$ first reaches either
$\Gamma_\varkappa$ or $\Gamma_{\varkappa_0}$, and the
expected time for this to happen is bounded by Lemma \ref{perturbed expectation time}. Since $\gamma_1<0$, this term is
$o(\varkappa^{\gamma_1})$.

By the strong Markov property,
\[
\mathrm{E}_x(\tau^\varepsilon(\Gamma_\varkappa)^2)=
2\mathrm{E}_x
\int_0^{\tau^\varepsilon(\Gamma_\varkappa)}
\mathrm{E}_{X_s^\varepsilon}
\tau^\varepsilon(\Gamma_\varkappa)\,ds\leq
2c_3\varkappa^{\gamma_1}
\mathrm{E}_x\tau^\varepsilon(\Gamma_\varkappa)
\leq 2c_3^2\varkappa^{2\gamma_1},\quad x\in K.
\]
This proves (\ref{variance bound}).
\end{proof}

\begin{lemma}
\label{square of the exit time}
There is $C>0$ such that
$$
\mathrm{E}_x(T_i^\varepsilon)^2\leq C (\mathrm{E}_x T_i^\varepsilon)^2,\quad i\in M,
$$
uniformly in $x\in K$, for all $\varepsilon$ sufficiently small.
\end{lemma}
\noindent
\begin{proof} We have an estimate on $\mathrm{E}_x T^\varepsilon_i$ from below by Lemma~\ref{expectation exit time}. We also have an estimate from below on the probability of the event $E_i$ from Lemma \ref{proportion exit probability}. Thus, it remains to estimate $\mathrm{E}_x ((\tau^{\varepsilon}(\partial D))^2\chi_{\{E_i\}})$  from above. Let us fix some $\varkappa$ sufficiently small so that the estimates on $I_n^{\varkappa,\varepsilon},J_n^{\varkappa,\varepsilon},s_n^{\varkappa,\varepsilon}$ from Lemma \ref{expectation exit time} and \ref{exponential exit time} hold for all $\varepsilon$ sufficiently small. Note that for each fixed $\varkappa$ sufficiently small, by Lemma 3.7 in \cite{Metastability} and by the definition of $t_n^{\varkappa,\varepsilon}$, there is $c_1>0$ such that $$
\mathrm{E}_x(t_n^{\varkappa,\varepsilon})^2\leq c_1,\quad x\in K,\, n\geq 1,
$$
for all $\varepsilon$ sufficiently small. Using these estimates and representing the time until the exit from $D$ as a sum of contributions from different excursions between $K$ and $\Gamma_\varkappa$ and back, it is not difficult to see that it is sufficient to show that for each fixed $\varkappa$ sufficiently small, there is $c_2>0$ such that
\[
\mathrm{E}_x(
(t_{n+1}^{\varkappa,\varepsilon})^2
\chi_{\{I_{n+1}^{\varkappa,\varepsilon}=0\}}
|\mathcal{G}_n^{\varkappa,\varepsilon})
\leq c_2
\frac{p_i(\varkappa,\varepsilon)}
     {p(\varkappa,\varepsilon)}
\]
on $\{I_n^{\varkappa,\varepsilon}=1,\,
J_{n+1}^{\varkappa,\varepsilon}=i\}$,
uniformly in $x\in K$, $n\geq0$, and $i\in M$. This can be done using arguments similar to those used in the proof of (\ref{last step exit time assumption}). The required conditional second-moment estimate follows
from the case $r=2$ of (\ref{last step exit time})
and the strong Markov property at
$\sigma_{n+1}^{\varkappa,\varepsilon}$.
\end{proof}

\section{Transition probabilities and exit times starting from the surfaces}
\label{Exit from surfaces}
In this section, we describe the behavior of the process $X_t^\varepsilon$, assuming that it starts on a surface $S$. (We temporarily drop the subscript $i$ from the notation, since we are talking about a single surface here.)  Note that the leading term of each of the coefficients of $L$ is symmetric near~$S$. Let us introduce the necessary notation. Similarly to the notation in Section~\ref{procbeh}, in a small two-sided neighborhood of $S$, each point $x$ can be uniquely identified with the pair $(y,z)$, where $y$ is the nearest point on the surface and $z$ is the signed distance to $S$, i.e., $z=\mathrm{ dist}(x,S)$ on one side of $S$ (this side is chosen a priori) and $z=-\mathrm{ dist}(x,S)$ on the other side of $S$. For each $\varkappa>0$ sufficiently small, let $\Gamma_{-\varkappa}:=\{(y,z)|\varphi(y)^{\frac{1}{\gamma}}z=-\varkappa\}$ and $V_{-\varkappa_1,\varkappa_2}:= \{(y,z)|-\varkappa_1\leq \varphi(y)^{\frac{1}{\gamma}}z\leq \varkappa_2\},\, \varkappa_1,\varkappa_2>0$.

\begin{lemma}
\label{equiprobability transition lemma} There is the following limit:
$$
\lim_{\varepsilon \downarrow 0}\mathrm{P}_x(X^{\varepsilon}_{\tau(\Gamma_{\varkappa_0}^i\cup \Gamma_{-\varkappa_0}^i)} \in \Gamma_{\varkappa_0}^i)= \frac{1}{2},
$$
uniformly in $x\in S_i$.
\end{lemma}
\begin{proof}
Note that we have
$$
V_{-\varkappa_0,\varkappa_0}^i= V_{-\varkappa_0, -R\varepsilon}^i \bigcup V_{-R\varepsilon, R \varepsilon}^i\bigcup V_{R\varepsilon,\varkappa_0}^i,
$$
for each $R>0$ and all $\varepsilon$ sufficiently small (depending on $R$).

Similarly to (\ref{second part transition}), for each $\delta>0$, we have
$$
\sup_{x\in \Gamma_{R\varepsilon}^i} \mathrm{P}_x (X_{\tau(S_i\bigcup \Gamma_{\varkappa_0}^i)}^\varepsilon\in S_i)<\delta,\quad \sup_{y\in \Gamma_{-R\varepsilon}^i} \mathrm{P}_y (X_{\tau( \Gamma_{-\varkappa_0}^i\bigcup S_i)}^\varepsilon\in S_i)<\delta,
$$
for all $R$ sufficiently large, for all $\varepsilon$ sufficiently small (depending on $R$). Thus, it is sufficient to show that for each $R$ fixed,
$$
\lim_{\varepsilon\downarrow 0} \mathrm{P}_x (X_{\tau(\Gamma_{-R\varepsilon}^i\bigcup \Gamma_{R\varepsilon}^i)}^\varepsilon\in \Gamma_{R\varepsilon}^i) = \frac{1}{2},\quad x\in S_i.
$$
Let $u^\varepsilon(y,z) := \mathrm{P}_{(y,\varepsilon z)}(X_{\tau(\Gamma_{-R\varepsilon}^i\bigcup \Gamma_{R\varepsilon}^i)}^\varepsilon\in \Gamma_{R\varepsilon}^i)$ and $u^\varepsilon_-(y,z)=\mathrm{P}_{(y,-\varepsilon z)}(X_{\tau(\Gamma_{-R\varepsilon}^i\bigcup \Gamma_{R\varepsilon}^i)}^\varepsilon \in \Gamma_{-R\varepsilon}^i)$ for $(y,z)\in V_{-R,R}^i$. Note that $u^\varepsilon$ and $u^\varepsilon_-$ satisfy elliptic equations in $V^i_{-R,R}$ with the same boundary conditions, i.e., $u^\varepsilon|_{\Gamma_R^i} = u^\varepsilon_-|_{\Gamma_R^i} =1$ and $u^\varepsilon|_{\Gamma_{-R}^i} = u^\varepsilon_-|_{\Gamma_{-R}^i} =0$. The coefficients in the equations for $u^\varepsilon$ and $u^\varepsilon_-$ are asymptotically close. Therefore, by a priori estimate for elliptic equations,  
$$
\lim_{\varepsilon\downarrow 0}\|u^\varepsilon-u^\varepsilon_-\|_{L^\infty(S_i)} = 0,
$$
which completes the proof.
\end{proof}

For two adjacent domains $D_i$ and $D_j$ with the common boundary $S$ and the respective compact sets $K_i$ and $K_j$, let $\tau_{ij}^\varepsilon := \inf \{t>0| X_t^{\varepsilon}\in K_i\bigcup K_j\}.$
\begin{lemma}
\label{time to decide the destination}
We have the following estimates:
$$
\mathrm{E}_x(\tau_{ij}^\varepsilon)=O(|\ln (\varepsilon)|),\quad \mathrm{E}_x((\tau_{ij}^\varepsilon)^2) = O(|\ln (\varepsilon)|^2),
$$
uniformly for $x\in S$ as $\varepsilon\downarrow 0$.
\end{lemma}
\begin{proof}
Note that for starting points in $V_{-\varkappa_0,\varkappa_0}$, following the definition, we have
\[
\tau_{ij}^\varepsilon
=\tau^\varepsilon
  (\Gamma_{-\varkappa_0}\cup\Gamma_{\varkappa_0}).
\]
By Lemma~\ref{transition probability to surface}, we can choose $R$ sufficiently large such that
\[
\sup_{x\in\{\Gamma_{-R\varepsilon}\cup\Gamma_{R\varepsilon}\}}
\mathrm{P}_x(
\tau^\varepsilon(S)<\tau_{ij}^\varepsilon)\leq\frac12
\]
for all sufficiently small $\varepsilon$.

Using the change of variable $z\mapsto z/\varepsilon$ and the last remark in Section \ref{procbeh}, we obtain that the mean
time to reach
$\Gamma_{-R\varepsilon}\cup\Gamma_{R\varepsilon}$ from $S$ is uniformly bounded.
Moreover, Lemma~\ref{perturbed expectation time}
gives
\[
\sup_{x\in V_{-\varkappa_0,\varkappa_0}}
\mathrm{E}_x\tau^\varepsilon
(S\cup\Gamma_{-\varkappa_0}\cup\Gamma_{\varkappa_0})
\leq c(1+|\ln\varepsilon|).
\]
Let us consider successive excursions between $S$ and
$\Gamma_{-R\varepsilon}\cup\Gamma_{R\varepsilon}$. We stop the excursions when $\Gamma_{-\varkappa_0}\bigcup \Gamma_{\varkappa_0}$ is reached.
Using the strong Markov property, a simple calculation shows that there is $c_1>0$ such that
\[
\sup_{x\in S}\mathrm{E}_x\tau_{ij}^\varepsilon
\leq c_1(1+|\ln\varepsilon|)
     \sum_{n=1}^{\infty}2^{1-n}
\leq 4c_1(1+|\ln\varepsilon|).
\]
Starting elsewhere in $V_{-R\varepsilon,R\varepsilon}$, first stop upon
reaching $S$ or an outer layer. The preceding
estimates and the strong Markov property then imply
\[
\sup_{x\in V_{-\varkappa_0,\varkappa_0}}
\mathrm{E}_x\tau_{ij}^\varepsilon
\leq c_2(1+|\ln\varepsilon|).
\]
This implies that
\[
\mathrm{E}_x((\tau_{ij}^\varepsilon)^2)=2\mathrm{E}_x\int_0^{\tau_{ij}^\varepsilon}
       \mathrm{E}_{X_s^\varepsilon}
       \tau_{ij}^\varepsilon\,ds\leq
2c_2(1+|\ln\varepsilon|)
\mathrm{E}_x\tau_{ij}^\varepsilon
\leq 2c_2(1+|\ln\varepsilon|)^2,
\]
uniformly for $x\in S$.
\end{proof}

\section{Metastable behavior of the  diffusion process}
\label{vague version main result}
{ In this section, we prove
Theorem~\ref{main theorem}.
We establish the metastable limits, identify their supports,
and derive explicit formulas for the limiting weights using
the adjacency tree defined below.}

We begin by fixing some notation. As in
Theorem~\ref{main theorem}, order the critical values as
\[
\gamma_{m+1}:=-\infty<\gamma_m\leq\cdots
\leq\gamma_1<0=:\gamma_0.
\]
For a domain $D_i$, let $M_i$ be the set of indices
$j$ such that $D_j$ is adjacent to $D_i$.
For $j\in M_i$, let $S_{ij}$ be their common boundary
surface and let $\gamma_{ij}$ be its critical value. Let
\[
\gamma_i^*:=\max_{j\in M_i}\gamma_{ij},
\qquad
l_i:=\{j\in M_i:\gamma_{ij}=\gamma_i^*\}.
\]
Let $G:=\bigcup_{j=1}^{m+1}K_j$.
We regard the sets $K_j$ as the vertices of the
adjacency tree: $K_i$ and $K_j$ are joined by an edge
if $D_i$ and $D_j$ share the surface $S_{ij}$.
This edge is assigned the critical value $\gamma_{ij}$.
We identify each subtree with the union of its
corresponding sets $K_j$.
For two disjoint subtrees $G_1,G_2$, write
$G_1\sim G_2$ if an edge joins a vertex of $G_1$
to a vertex of $G_2$.

For $\gamma\in[-\infty,0]$, let $G^\gamma(i)$
be the connected component containing $K_i$
after deleting all edges with critical value
less than or equal to $\gamma$. In particular,
$G^{\gamma_i^*}(i)=K_i$ and $G^{-\infty}(i)=G$. { Note that in the notation of Theorem~\ref{main theorem},
\[
\mathcal I_k(i)
=\{j:K_j\subseteq G^{\gamma_k}(i)\},
\qquad
G^{\gamma_k}(i)
=\bigcup_{j\in\mathcal I_k(i)}K_j.
\]
Thus, the sets $\bigcup_{j\in\mathcal I_k(i)}D_j$ describe
the supports of the metastable distributions
in Theorem~\ref{main theorem}.

To prove Theorem~\ref{main theorem}, we use the
semi-Markov framework of \cite{koralov2024metastable}.
We first construct a semi-Markov process
$Y_t^\varepsilon$ associated with $X_t^\varepsilon$. The process $Y^\varepsilon_t$ will be shown to satisfy the assumptions in \cite{koralov2024metastable}, and so the metastability results in \cite{koralov2024metastable} apply.We then identify its limiting distributions
and their weights using the adjacency tree,
and transfer these results to $X_t^\varepsilon$.}
\begin{definition}
\label{semi Markov process Yt}
Define a semi-Markov process $Y_t^\varepsilon$ with the state space $G$ as follows: $\xi^\varepsilon_0 = 0$, $Y^\varepsilon_0 = X^\varepsilon_0\in G$, and, for $n \geq 1$, given that $X_{\xi_{n-1}^\varepsilon}^\varepsilon \in K_i$,
$$
\xi_n^\varepsilon := \underset{t}{\inf}\{t> \xi_{n-1}^\varepsilon | X_{t}^\varepsilon\in K_j,\, j\neq i\},\quad Y_t^\varepsilon= X_{\xi_{n-1}^\varepsilon}^\varepsilon,\quad t\in [\xi_{n-1}^\varepsilon,\xi_n^\varepsilon).
$$
\end{definition}
Let us recall the list of properties of a semi-Markov process that ensure that it exhibits meta-stable behavior (according to the results in Section 4 in \cite{koralov2024metastable} that concerns the situation when the transition times depend very weakly on the target set). As applied to our situation, the properties concern the transition times and the transition probabilities
between the sets $K_i$. Informally, for two adjacent domains $D_i$ and $D_j$, by the strong Markov property, the transition from $K_i$ to $K_j$ can be studied by examining the transition from $K_i$ to $S_{ij}$ and then from $S_{ij}$ to $K_j$. The time and probability for the first part of the transition have been studied in Section \ref{Probability and expectation time for transition in one domain}, and for the second part they have been studied in Section \ref{Exit from surfaces}.

1) For two adjacent domains $D_i$ and $D_j$, by Lemma \ref{proportion exit probability} and Lemma \ref{equiprobability transition lemma}, there are $C_{ij},\, \rho_{ij}>0$ such that
\begin{equation}
\label{proportion transition probability for semi markov process}
\mathrm{P}_x(Y_{\xi_1^\varepsilon}^\varepsilon\in K_j)\sim \frac{C_{ij}\rho_{ij}}{\sum_{k \in l_i}C_{ik}\rho_{ik}}\varepsilon^{-\gamma_{i j}+\gamma_{i}^*},\quad x\in K_i,\, j\in M_i,\,\quad \mathrm{ as}\,\, \varepsilon\downarrow 0.
\end{equation}
Moreover, the numbers $C_{ij}$ can be determined in terms of the coefficients of $L$ in $D_i$ and the numbers $\rho_{ij}$ can be determined in terms of the coefficients of $L,\widetilde L$ in an arbitrarily small neighborhood of $S_{ij}$.

2) Let $\xi_1^\varepsilon(j)$ be the random time $\xi_1^\varepsilon$ conditioned on the event $\{Y_{\xi_1^\varepsilon}^\varepsilon\in K_j\}$. By Lemma~\ref{expectation exit time}, Lemma \ref{equiprobability transition lemma}, and Lemma \ref{time to decide the destination}, there is $C_i>0$ such that 
\begin{equation}
\label{condition expectation exit time for semi markov process}
\mathrm{E}_x(\xi_1^\varepsilon(j)) \sim \frac{C_i}{\sum_{k\in l_i}C_{ik}\rho_{ik}} \varepsilon^{\gamma_{i}^*},\quad x\in K_i,\, j\in M_i,\quad \mathrm{ as}\,\, \varepsilon\downarrow 0,
\end{equation}
where $C_i$ can be determined in terms of the coefficients of $L$ in $D_i$ and $C_{ij},\rho_{ij}$ are the same as in (\ref{proportion transition probability for semi markov process}). 

3) By Lemma \ref{square of the exit time}, Lemma \ref{equiprobability transition lemma}, and Lemma \ref{time to decide the destination}, there is $C'_i>0$ such that
$$
\mathrm{E}_x(\xi_1^\varepsilon(j)^2) \leq C'_i\varepsilon^{2\gamma_{i}^*},\quad x\in K_i,\, j\in M_i,\quad \mathrm{ as}\,\, \varepsilon\downarrow 0.
$$

4) By Lemma \ref{exponential exit time}, Lemma \ref{equiprobability transition lemma}, and Lemma \ref{time to decide the destination}, for each $t\in [0,\infty)$, we have
$$
\lim_{\varepsilon\downarrow 0}\mathrm{P}_x((\frac{C_i}{\sum_{k\in l_i}C_{ik}\rho_{ik}}\varepsilon^{\gamma_{i}^*})^{-1}\cdot \xi_1^\varepsilon(j) \geq t) = e^{-t},\quad x\in K_i,\, j\in M_i.
$$

5) All the transition probabilities and the expectations of transition times (the left hand sides of (\ref{proportion transition probability for semi markov process}) and (\ref{condition expectation exit time for semi markov process})) behave as powers of $\varepsilon$. Therefore, any ratio of two finite products of such quantities has a limit (which may be equal to $+\infty$). This means that the complete asymptotic regularity assumption from \cite{koralov2024metastable} is met.  

Our properties 1)-5) imply that the Assumptions (a)-(f) in Section $4$ in \cite{koralov2024metastable} are satisfied.
Consequently, by Theorem 4.4 in \cite{koralov2024metastable}, there is a finite sequence of time scales such that there is a limit for the distribution of $Y^\varepsilon_{t(\varepsilon)}$, whenever $t(\varepsilon)$ is between these critical time scales. 
Our goal is not simply to prove that the abstract result applies to our situation (which is what we just accomplished), but to specify the critical time scales and the corresponding metastable distributions for $Y_t^\varepsilon$, as stated in the following Lemma.
\begin{lemma}\label{vague result of limiting measure for process Yt on compact sets}
{ There exist positive numbers $w_1,\ldots,w_{m+1}$, determined
solely by the coefficients of $L$, with
$\sum_{j=1}^{m+1}w_j=1$, such that the following holds.
For all indices $i,k$ and time scales $t(\varepsilon)$
specified in Theorem \ref{main theorem},}
\[
 \lim_{\varepsilon\downarrow0}
 \mathrm P_x\bigl(Y_{t(\varepsilon)}^\varepsilon\in K_j\bigr)
 =c_j(i,k),\qquad 1\le j\le m+1,
\]
uniformly in $x\in G^{\gamma_k}(i)$, where
\[
 c_j(i,k)=
 \begin{cases}
  \displaystyle\frac{w_j}{w(G^{\gamma_k}(i))},
   &K_j\subseteq G^{\gamma_k}(i),\\[6pt]
  0,&K_j\not\subseteq G^{\gamma_k}(i),
 \end{cases}
 \qquad
 w(A):=\sum_{K_a\subseteq A}w_a.
\]
In particular, when $k=m+1$, the coefficients are
$c_j(i,m+1)=w_j$ and do not depend on $i$.
\end{lemma}

\begin{proof}
For adjacent vertices
$K_a$ and $K_b$, let $p_{ab}^\varepsilon$ and
$\theta_a^\varepsilon$ be the transition probability and mean
holding time of $Y^\varepsilon$, respectively. By properties
1)--2),
\begin{equation}\label{eq:tree-effective-ratio}
 \frac{p_{ab}^\varepsilon}{\theta_a^\varepsilon}
 \sim q_{ab}\varepsilon^{-\gamma_{ab}},
 \qquad q_{ab}:=\frac{C_{ab}\rho_{ab}}{C_a}>0.
\end{equation}
Since the adjacency graph is a tree, there are unique
positive numbers $w_1,\ldots,w_{m+1}$ satisfying
\begin{equation}\label{eq:tree-balance}
 \sum_{a=1}^{m+1}w_a=1,\qquad
 w_aq_{ab}=w_bq_{ba}\quad\text{whenever }K_a\sim K_b.
\end{equation}
Indeed, after fixing the weight at one vertex, the weights at
all other vertices are determined successively along the unique
paths in the tree.
Moreover, since $\rho_{ab}=\rho_{ba}$ for adjacent vertices,
\[
 \frac{q_{ab}}{q_{ba}}
 =\frac{C_{ab}\rho_{ab}/C_a}{C_{ba}\rho_{ba}/C_b}
 =\frac{C_{ab}C_b}{C_{ba}C_a}.
\]
The constants $C_a$ and $C_{ab}$ depend only on the coefficients
of $L$. Thus the ratios $w_b/w_a=q_{ab}/q_{ba}$, and hence
all the normalized weights $w_a$, are independent of the
perturbation $\widetilde L$.

For a union $A$ of vertices, we use the notation $w(A)$
introduced in the statement of the lemma. We consider only $A=G^{\gamma}(i)$ defined at the beginning of
Section~\ref{vague version main result}. Every such $A$ is
connected, and every edge inside a proper set $A$ has a
strictly larger critical value than every edge leaving $A$.
For $A\ne G$, define
\[
 \Theta_A^\varepsilon
 :=\inf\{t\geq0:Y_t^\varepsilon\notin A\},\qquad
 H_A^\varepsilon
 :=\sum_{\substack{K_a\subseteq A,\ K_b\not\subseteq A\\K_a\sim K_b}}
       w_aq_{ab}\varepsilon^{-\gamma_{ab}}.
\]
Write $\gamma^+(A)$ for the smallest critical value among
edges inside $A$, with $\gamma^+(A):=0$ if $A$ is a singleton.
For $A\ne G$, write $\gamma^-(A)$ for the largest critical
value among edges leaving $A$. Thus
$\gamma^+(A)>\gamma^-(A)$ whenever $A\ne G$.

We prove the following assertions by induction on the number
of vertices in $A$. For $A\ne G$, we have
\begin{align}
 \mathrm P_x\bigl(Y_{\Theta_A^\varepsilon}^\varepsilon
                    \in K_b\bigr)
 &\sim \frac{w_aq_{ab}\varepsilon^{-\gamma_{ab}}}
               {H_A^\varepsilon},
 &&K_a\subseteq A,\ K_b\not\subseteq A,\ K_a\sim K_b,
 \label{eq:tree-set-exit-prob}\\
 \mathrm E_x\Theta_A^\varepsilon
 &\sim \frac{w(A)}{H_A^\varepsilon},
 \label{eq:tree-set-exit-mean}.
\end{align}
Here and below, unless specified otherwise, the process
is assumed to start at $x\in A$, and all asymptotic
statements hold uniformly over such starting points. The same mean asymptotics holds after conditioning on any
specified exit destination. We include in the induction the
uniform conditional second-moment bounds and exponential
limits for exit times normalized by their means.

If $A=K_a$, then
\[
 H_A^\varepsilon
 =w_a\sum_{K_b\sim K_a}q_{ab}\varepsilon^{-\gamma_{ab}}.
\]
Thus \eqref{eq:tree-set-exit-prob}--\eqref{eq:tree-set-exit-mean}
follow from properties 1)--2) and
\eqref{eq:tree-effective-ratio}. The conditional moment bounds
and exponential limits follow from properties 3)--4).

Now suppose that $A$ contains at least two vertices, and
assume the assertions for every proper subset
$B\subsetneq A$ of the form $G^\gamma(j)$.
Let $\gamma:=\gamma^+(A)$. Deleting from $A$ all edges
with critical value $\gamma$ leaves connected components
$A_1,\ldots,A_r$, where $r\geq2$. Each $A_s$ is of the form
$G^\gamma(j_s)$ and is a proper subset of $A$.
Every edge joining two distinct $A_s$'s has critical value
$\gamma$, whereas every edge leaving $A$ has a strictly
smaller critical value.

For $s\ne u$, define
\[
 b_{su}:=
 \sum_{\substack{K_a\subseteq A_s,\ K_b\subseteq A_u\\K_a\sim K_b}}
     w_aq_{ab},\qquad
 d_s:=\sum_{u\ne s}b_{su}.
\]
The numbers $d_s$ are strictly positive. By
\eqref{eq:tree-balance}, $b_{su}=b_{us}$. Since
$H_{A_s}^\varepsilon\sim d_s\varepsilon^{-\gamma}$,
the induction hypothesis gives
\begin{equation}\label{eq:tree-lower-level-means}
 \mathrm E_x\Theta_{A_s}^\varepsilon
 \sim \frac{w(A_s)}{d_s}\varepsilon^\gamma,
 \qquad x\in A_s.
\end{equation}
Moreover, the probability that the next component visited is $A_u$
converges, uniformly over starting points in $A_s$, to
$P^0_{su}:=b_{su}/d_s$ for $s\ne u$; set $P^0_{ss}:=0$.
Note that the probability of leaving $A$ at this
transition is $o(1)$. The matrix $P^0$ is irreducible and
has invariant probabilities
\begin{equation}\label{eq:tree-lower-level-stationary}
 \lambda_s=\frac{d_s}{\sum_{v=1}^r d_v}.
\end{equation}
Indeed, $\lambda_sP^0_{su}=b_{su}/\sum_vd_v
=\lambda_uP^0_{us}$.

Suppose that $A\ne G$. By the induction hypothesis,
the conditional holding times for transitions out of each
$A_s$ have the common mean asymptotics
\eqref{eq:tree-lower-level-means}, uniformly over their
destinations, with the required second-moment bounds and
exponential limits. Their transition probabilities have
power-law asymptotics, and $P^0$ is irreducible. Thus the
finite-state rare-exit estimates (as in Lemma 5.2 and Theorems
6.5, 6.7, and 6.9 of \cite{koralov2024metastable}) apply to
the process of successive visits to $A_1,\ldots,A_r$.

In other words, for the purpose of calculating the asymptotics in \eqref{eq:tree-set-exit-prob}--\eqref{eq:tree-set-exit-mean}, we can imagine that $A_1,...,A_r$ are individual points rather than sets and the transitions between them and from them to the outside of $A$ are governed by the asymptotic expressions we obtained above.

To specify the averaged exit probability from $A$, we can choose some
$x_s\in A_s$, $1\leq s\leq r$, and set
\[
 h_A^\varepsilon
 :=\sum_{s=1}^r\lambda_s\,
 \mathrm P_{x_s}\bigl(
    Y_{\Theta_{A_s}^\varepsilon}^\varepsilon\notin A
 \bigr).
\]
The following asymptotics is independent of the choices of
$x_s$, by the uniform exit estimates in the induction
hypothesis:
\begin{equation}\label{eq:tree-averaged-escape}
 h_A^\varepsilon
 \sim \frac{\varepsilon^\gamma}{\sum_{v=1}^r d_v}
 \sum_{\substack{K_a\subseteq A,\ K_b\not\subseteq A\\a\sim b}}
          w_aq_{ab}\varepsilon^{-\gamma_{ab}}
 =\frac{\varepsilon^\gamma H_A^\varepsilon}
        {\sum_{v=1}^r d_v}.
\end{equation}
The average holding time in the sets $A_1,...,A_r$ is
asymptotic to
\begin{equation}\label{eq:tree-averaged-holding}
 \sum_{s=1}^r\lambda_s\frac{w(A_s)}{d_s}\varepsilon^\gamma
 =\frac{\varepsilon^\gamma w(A)}{\sum_{v=1}^r d_v}.
\end{equation}
Note that the expected
number of visits to $A_s$ before exiting $A$
is asymptotic to $\lambda_s/h_A^\varepsilon$. This gives
\[
\mathrm E_x\Theta_A^\varepsilon
\sim
\frac{1}{h_A^\varepsilon}
\sum_{s=1}^r
\lambda_s\frac{w(A_s)}{d_s}\varepsilon^\gamma
=
\frac{w(A)}{H_A^\varepsilon}.
\]
This proves \eqref{eq:tree-set-exit-mean}. For a boundary edge $K_a\sim K_b$, with
$K_a\subseteq A_s$ and $K_b\not\subseteq A$,
the stationary average probability of leaving $A$
through this edge in one transition is asymptotic to
\[
\lambda_s
\mathrm P_y\bigl(
Y_{\Theta_{A_s}^\varepsilon}^\varepsilon\in K_b
\bigr)
\sim
\frac{\varepsilon^\gamma}{\sum_{v=1}^r d_v}
w_aq_{ab}\varepsilon^{-\gamma_{ab}},
\qquad y\in A_s.
\]
Note that the probability
of eventually exiting $A$ through this edge is
asymptotic to the preceding quantity divided by
$h_A^\varepsilon$. Thus,
\[
\begin{aligned}
\mathrm P_x\bigl(
Y_{\Theta_A^\varepsilon}^\varepsilon\in K_b
\bigr)\sim
\frac{\varepsilon^\gamma
w_aq_{ab}\varepsilon^{-\gamma_{ab}}}
{(\sum_{v=1}^r d_v)h_A^\varepsilon}\sim
\frac{w_aq_{ab}\varepsilon^{-\gamma_{ab}}}
{H_A^\varepsilon}.
\end{aligned}
\]
This proves \eqref{eq:tree-set-exit-prob}. Similar estimates
also yield the claimed conditional mean, second-moment bounds,
and exponential exit-time limits. This completes the induction.

Next, for
\begin{equation}\label{eq:tree-time-window}
 \varepsilon^{\gamma^+(A)}\ll t(\varepsilon)
       \ll\varepsilon^{\gamma^-(A)},\qquad A\ne G,
\end{equation}
we prove that
\begin{equation}\label{eq:tree-set-equilibrium}
 \mathrm P_x(Y_{t(\varepsilon)}^\varepsilon\in K_j)
 \longrightarrow \frac{w_j}{w(A)},\qquad K_j\subseteq A.
\end{equation}
When $A=G$, the corresponding assertion holds under the sole
condition $\varepsilon^{\gamma^+(G)}\ll t(\varepsilon)$.

We prove \eqref{eq:tree-set-equilibrium}
by induction on the number of vertices in $A$. Since
$t(\varepsilon)\ll\varepsilon^{\gamma^-(A)}$ implies
$\mathrm P_x(\Theta_A^\varepsilon\leq t(\varepsilon))\to0$,
\eqref{eq:tree-set-equilibrium} holds for $A=K_a$.

Now suppose that $A$ contains at least two vertices,
and assume that \eqref{eq:tree-set-equilibrium}
holds for every subset of $A$ in the form of $G^{\gamma'}(j)$. Fix
$t(\varepsilon)$ satisfying \eqref{eq:tree-time-window}
when $A\ne G$, or satisfying
$\varepsilon^{\gamma^+(G)}\ll t(\varepsilon)$ when $A=G$.
For $A\ne G$, the exponential exit-time limit and
\eqref{eq:tree-set-exit-mean} imply
\begin{equation}\label{eq:tree-no-escape}
 \sup_{x\in A}\mathrm P_x(\Theta_A^\varepsilon\leq
                         t(\varepsilon))\longrightarrow0.
\end{equation}
Considering the process recording successive visits to
$A_1,\ldots,A_r$, by \eqref{eq:tree-lower-level-means} and
\eqref{eq:tree-lower-level-stationary}, we obtain (as in Theorem 7.4 of \cite{koralov2024metastable})
\begin{equation}\label{eq:tree-component-occupation}
 \mathrm P_x(Y_{t(\varepsilon)}^\varepsilon\in A_s)
 \longrightarrow
 \frac{\lambda_s w(A_s)/d_s}
      {\sum_{u=1}^r\lambda_u w(A_u)/d_u}
 =\frac{w(A_s)}{w(A)}.
\end{equation}
This establishes the
limiting probabilities of the sets $A_s$; we still need to
identify those of the individual vertices. Let us consider a non-singleton $A_s$.

Let us choose an intermediate time $u(\varepsilon)>0$ such that
\begin{equation}\label{eq:tree-intermediate-window}
 \max_{1\leq k\leq r}\varepsilon^{\gamma^+(A_k)}
 \ll u(\varepsilon)\ll\varepsilon^\gamma.
\end{equation}
In particular, $u(\varepsilon)=o(t(\varepsilon))$.
The induction hypothesis, applied to $A_s$, gives
\begin{equation}\label{eq:tree-fine-equilibrium}
 \sup_{y\in A_s}
 \sup_{u(\varepsilon)\leq v\leq2u(\varepsilon)}
 \left|\mathrm P_y(Y_v^\varepsilon\in K_j)
           -\frac{w_j}{w(A_s)}\right|\longrightarrow0,
 \qquad K_j\subseteq A_s.
\end{equation}
The uniformity in $v$ follows by contradiction: any
choice $v=v(\varepsilon)\in[u(\varepsilon),2u(\varepsilon)]$
is itself a time scale strictly between the internal and
outgoing transition scales of $A_s$.

Starting at time
$t(\varepsilon)-2u(\varepsilon)$, let
$\sigma^\varepsilon_s$ be the first renewal of the process
recording successive visits to the connected components
obtained by deleting the edges of critical value
$\gamma^+(A_s)$ from $A_s$ and visits to the complement of $A_s$. This renewal time is also a stopping
time for the original process. The transition times between the connected components of $A_s$ have order $\varepsilon^{\gamma^+(A_s)}$. Then it follows (see Lemma 7.3 of \cite{koralov2024metastable}) that, with probability close to one, conditioned on $Y_{t(\varepsilon)-2u(\varepsilon)}^\varepsilon\in A_s$, the renewal $\sigma^\varepsilon_s$ happens within time $u(\varepsilon)$ from $ t(\varepsilon)-2u(\varepsilon)$, and the process does not leave $A_s$ at this renewal time. Combining this with \eqref{eq:tree-component-occupation}  applied to
$t(\varepsilon)-2u(\varepsilon)$, we obtain
\begin{equation}\label{eq:tree-renewal-window}
 \mathrm P_x\bigl(
    Y_{t(\varepsilon)-2u(\varepsilon)}^\varepsilon\in A_s,
    \ \sigma_s^\varepsilon\leq t(\varepsilon)-u(\varepsilon),
    \ Y_{\sigma_s^\varepsilon}^\varepsilon\in A_s
   \bigr)
 \longrightarrow\frac{w(A_s)}{w(A)}.
\end{equation}
On the event in \eqref{eq:tree-renewal-window}, the time
remaining after $\sigma_s^\varepsilon$ lies in
$[u(\varepsilon),2u(\varepsilon)]$. The strong
Markov property at $\sigma_s^\varepsilon$ and
\eqref{eq:tree-fine-equilibrium} therefore give, for
$K_j\subseteq A_s$,
\[
 \lim_{\varepsilon\downarrow0}
 \mathrm P_x(Y_{t(\varepsilon)}^\varepsilon\in K_j)
 = \frac{w(A_s)}{w(A)}\frac{w_j}{w(A_s)}
 =\frac{w_j}{w(A)}.
\]
This proves \eqref{eq:tree-set-equilibrium}.
Note that if $A_s=K_j$ is a singleton, the same equality follows
directly from \eqref{eq:tree-component-occupation}. This completes the proof of the lemma.
\end{proof}

Now, we will prove Theorem \ref{main theorem} with the aid of Lemma \ref{vague result of limiting measure for process Yt on compact sets}.
\begin{proof}[Proof of Theorem~\ref{main theorem}]
{ Let us fix indices $i,k$ and a time scale
$t(\varepsilon)$ as in the theorem.
Let $w_j$ and $c_j(i,k)$ be given by
Lemma~\ref{vague result of limiting measure for process Yt on compact sets}. It suffices to consider starting points in $K_i$,
since any prescribed compact subset of $D_i$ can
be contained in $K_i$ by choosing $\varkappa_0$
sufficiently small and, in the unbounded case,
the truncating ball sufficiently large.}

Recall that $Y_t^\varepsilon$ is constructed from $X_t^\varepsilon$, and $\xi_n^\varepsilon$ are the renewal times for the process $Y_t^\varepsilon$, as defined in Definition \ref{semi Markov process Yt}. Let $\xi_-(s):= \sup \{\xi_n^\varepsilon|\xi_n^\varepsilon<s\}$. By Lemma \ref{vague result of limiting measure for process Yt on compact sets}, we have
\begin{equation}
\label{limiting invariant measure}
\lim_{\varepsilon\downarrow 0} \mathrm{P}_x(X_{\xi_-(t(\varepsilon))}^\varepsilon\in K_j) = c_j(i,k), \quad j\in \{1,...,m+1\},\quad x\in G^{\gamma_k}(i),
\end{equation}
where $c_j(i,k)>0$ if and only if $K_j\subseteq G^{\gamma_k}(i)$.

{ We first prove the metastable convergence in
Theorem~\ref{main theorem}.
It suffices to show that, for each fixed
$f\in C_b(\mathbb R^d)$,}
\begin{equation}
\label{test function result of limiting measure Xt}
\lim_{\varepsilon\downarrow 0}|\mathrm{E}_x f(X^\varepsilon_{t(\varepsilon)}) -\sum_{j=1}^{m+1}c_j(i,k)\int_{D_j} f(x)\mu_j(dx)|=0,\quad x\in G^{\gamma_k}(i).
\end{equation}
Recall that in Section \ref{Probability and expectation time for transition in one domain}, for the process $X_t^\varepsilon$ in a certain domain $D_i$, we have studied the excursions from $\Gamma_{\varkappa_0}$ to $\Gamma_{\varkappa_0/2}$ and back. Similarly, we now consider the excursions from $\overline{\Gamma}_{\varkappa_0}$ to $\overline{\Gamma}_{\varkappa_0/2}$ and back, where $\overline{\Gamma}_{\varkappa_0}$ denotes the union of the  layer $\Gamma_{\varkappa_0}$ in all domains. Let us define the corresponding random times $\overline{\sigma}_k^{\varepsilon},\overline{\tau}_k^{\varepsilon}$ for such excursions:
$$
\overline{\sigma}_0^{\varepsilon}=\overline{\tau}_0^{\varepsilon}= 0,\quad \overline{\sigma}_k^{\varepsilon} = \inf\{t>\overline{\tau}_{k-1}^{\varepsilon}| X_{t}^\varepsilon\in \overline{\Gamma}_{\varkappa_0/2}\},\, \overline{\tau}_{k}^\varepsilon= \inf\{t>\overline{\sigma}_{k}^\varepsilon|X_t^\varepsilon\in \overline{\Gamma}_{\varkappa_0}\},\quad k\geq 1.
$$
We have estimates on the expectations and variances of the excursion times from Section~\ref{Probability and expectation time for transition in one domain}: there exist $c, C>0$ such that $c \leq \mathrm{E}_x \overline{\tau}^\varepsilon_1,~\mathrm{ Var} \overline{\tau}^\varepsilon_1 \leq C,$ $x \in \overline{\Gamma}_{\varkappa_0}$ for all $\varepsilon$ sufficiently small. Such estimates imply that, with probability close to one, the process starting in a compact set will have a stopping time inside any sufficiently large time interval. More precisely, for each $\delta>0$, for each time scale $s(\varepsilon)$, we have
\begin{equation}
\label{T1 restriction}
\mathrm{P}_x(\{\overline{\tau}_k^{\varepsilon}\in [s(\varepsilon), s(\varepsilon)+T_1]\,\, \text{for some}\, k\geq 0\}) \geq 1-\delta,\quad x \in G^{\gamma_k}(i),
\end{equation}
for all $T_1$ sufficiently large, for all $\varepsilon$ sufficiently small. For a proof of this statement we refer to Lemma 7.3 in \cite{koralov2024metastable}. 

Recall that for the unperturbed process $X_t$, $\mathrm{Law}_xX_t$ converges to $\mu_j$ uniformly in $x\in K_j$ as $t\to \infty$ (see \cite{freidlin2023perturbations}, \cite{Metastability}, where this result was proved using Theorem 3.7 in \cite{hairer2010convergence}). Thus, for each $\delta>0$, by the proximity of $X_t^\varepsilon$ and $X_t$ on finite time intervals, we can find sufficiently large $T_2$ such that
\begin{equation}
\label{T2 restriction}
\sup_{t\in [T_2, T_2+T_1]}|\mathrm{E}_x f(X_t^\varepsilon) - \int_{D_j} f(x)d\mu_j|<\delta,\quad x\in K_j,\quad j\in \{1,...,m+1\},
\end{equation}
 for all $\varepsilon $ sufficiently small (depending on $T_1,T_2$). 
 
Let $s(\varepsilon)=t(\varepsilon)-T_1-T_2$ and
$b(\varepsilon)=s(\varepsilon)+T_1=t(\varepsilon)-T_2$.
Let $\overline\tau^\varepsilon$ be the first of the stopping times
$\overline\tau_n^\varepsilon$ in
$[s(\varepsilon),b(\varepsilon)]$, if one exists, and put
$\overline\tau^\varepsilon=b(\varepsilon)$ otherwise.
Let us denote the event that such a return occurs by $A_\varepsilon$.
By (\ref{T1 restriction}),
$\mathrm P_x(A_\varepsilon^c)<\delta$.

For each fixed $T>0$, the invariance of the domains for the
unperturbed process $X_t$ and its uniform proximity to $X^\varepsilon$
on finite time intervals imply that
\[
q_\varepsilon(T):=
\max_{1\leq j\leq m+1}\sup_{y\in K_j}
\mathrm P_y(\xi_1^\varepsilon\leq T)
\to 0.
\]
Note that by Definition~\ref{semi Markov process Yt},
$X_{\overline\tau^\varepsilon}^\varepsilon$ and
$Y_{\overline\tau^\varepsilon}^\varepsilon$ belong to the same $K_j$,
whether or not $\overline\tau^\varepsilon$ is a renewal time.
The strong Markov property therefore gives
\[
\left|
\mathrm P_x(X_{\overline\tau^\varepsilon}^\varepsilon\in K_j)
-\mathrm P_x(Y_{b(\varepsilon)}^\varepsilon\in K_j)
\right|
\leq\delta+q_\varepsilon(T_1).
\]
Since $b(\varepsilon)\sim t(\varepsilon)$,
Lemma~\ref{vague result of limiting measure for process Yt on compact sets}
implies that, for all sufficiently small $\varepsilon$,
\begin{equation}
\label{excursion constant}
\left|
\mathrm P_x(X_{\overline\tau^\varepsilon}^\varepsilon\in K_j)
-c_j(i,k)
\right|<3\delta,
\qquad j\in\{1,\ldots,m+1\},
\end{equation}
uniformly in $x\in G^{\gamma_k}(i)$.

On $A_\varepsilon$, the remaining time
$t(\varepsilon)-\overline\tau^\varepsilon$ belongs to
$[T_2,T_2+T_1]$.
Thus, (\ref{T2 restriction}) and the strong Markov property imply
\[
\left|
\mathrm E_x f(X_{t(\varepsilon)}^\varepsilon)
-\sum_{j=1}^{m+1}
\mathrm P_x(X_{\overline\tau^\varepsilon}^\varepsilon\in K_j)
\int_{D_j} f\,d\mu_j
\right|
\leq\delta+2\|f\|_\infty\delta,
\]
where the second term bounds the contribution of
$A_\varepsilon^c$.
Combining this with (\ref{excursion constant}) and then letting
$\delta\downarrow0$ proves
(\ref{test function result of limiting measure Xt}). { This proves the metastable convergence
asserted in Theorem~\ref{main theorem}.}

It remains to prove the convergence of the invariant
measures. Since $G^{\gamma_{m+1}}(i)=G$,
Lemma~\ref{vague result of limiting measure for process Yt on compact sets}
implies that $c_j:=c_j(i,m+1)>0$ does not depend on $i$.
{ By the metastable convergence we have just proved}, for every time scale
$t(\varepsilon)\gg\varepsilon^{\gamma_m}$ and
$f\in C_b(\mathbb{R}^d)$, we have
\begin{equation}
\label{test function result of the largest time scale}
\mathrm{E}_x f(X_{t(\varepsilon)}^\varepsilon)
\longrightarrow
\sum_{j=1}^{m+1}c_j\int_{D_j}f\,d\mu_j,
\end{equation}
uniformly in $x\in G$.

Note that for each sufficient small $\varepsilon>0$,
Assumption~(e) and non-degeneracy imply that
$X^\varepsilon_t$ has a unique invariant probability
measure $\mu^\varepsilon$ and that
$\mathrm{Law}_x(X_t^\varepsilon)\to\mu^\varepsilon$
as $t\to\infty$. Therefore, for every $f\in C_b(\mathbb{R}^d)$ and the starting point $x\in G$, and for each sufficiently small $\varepsilon$, we can choose $u(\varepsilon)\geq\varepsilon^{\gamma_m-1}$
large enough so that
\[
\left|
\mathrm{E}_{x}f(X_{u(\varepsilon)}^\varepsilon)
-\int f\,d\mu^\varepsilon
\right|\leq\varepsilon.
\]
Since
$u(\varepsilon)\gg\varepsilon^{\gamma_m}$, applying (\ref{test function result of the largest time scale})
therefore gives
\[
\int f\,d\mu^\varepsilon
\longrightarrow
\sum_{j=1}^{m+1}c_j\int_{D_j}f\,d\mu_j.
\]
Since $f$ was arbitrary, the invariant measures
converge weakly to the stated mixture.
\end{proof}
\section*{Acknowledgments}
We are very grateful to Mark Freidlin for important discussions.

\section*{Declaration on the use of AI} The authors consulted with AI (Claude Opus 5.5) to check some of the proofs, and used it to simplify one of the arguments. All mathematical content is the authors’ own. 

\section*{Funding}
Leonid Koralov was supported by the NSF grant DMS-2307377.
Both authors were supported by the Simons Foundation
Grant MP-TSM-00002743.

\bibliography{bibfile}
\end{document}